\documentclass[english]{article}
\usepackage[T1]{fontenc}
\usepackage[latin9]{inputenc}
\usepackage{geometry}
\usepackage{color}
\usepackage{array}
\usepackage{float}
\usepackage{mathtools}
\usepackage{url}
\usepackage{bm}
\usepackage{amsmath}
\usepackage{amsthm}
\usepackage{amssymb}
\usepackage{graphicx}

\makeatletter

\newcommand{\noun}[1]{\textsc{#1}}
\providecommand{\tabularnewline}{\\}
\floatstyle{ruled}
\newfloat{algorithm}{tbp}{loa}
\providecommand{\algorithmname}{Algorithm}
\floatname{algorithm}{\protect\algorithmname}

\theoremstyle{plain}
\newtheorem{thm}{\protect\theoremname}
  \theoremstyle{plain}
  \newtheorem{lem}[thm]{\protect\lemmaname}
  \theoremstyle{definition}
  \newtheorem{defn}[thm]{\protect\definitionname}
  \theoremstyle{plain}
  \newtheorem{cor}[thm]{\protect\corollaryname}
  \theoremstyle{remark}
  \newtheorem{claim}[thm]{\protect\claimname}

\usepackage{algorithmic}

\makeatother

\usepackage{babel}
  \providecommand{\claimname}{Claim}
  \providecommand{\corollaryname}{Corollary}
  \providecommand{\definitionname}{Definition}
  \providecommand{\lemmaname}{Lemma}
\providecommand{\theoremname}{Theorem}

\begin{document}
\global\long\def\R{\mathbb{R}}

\global\long\def\C{\mathbb{C}}

\global\long\def\N{\mathbb{N}}

\global\long\def\e{{\mathbf{e}}}

\global\long\def\et#1{{\e(#1)}}

\global\long\def\ef{{\mathbf{\et{\cdot}}}}

\global\long\def\a{{\mathbf{a}}}

\global\long\def\x{{\mathbf{x}}}

\global\long\def\xt#1{{\x(#1)}}

\global\long\def\xf{{\mathbf{\xt{\cdot}}}}

\global\long\def\d{{\mathbf{d}}}

\global\long\def\w{{\mathbf{w}}}

\global\long\def\b{{\mathbf{b}}}

\global\long\def\u{{\mathbf{u}}}

\global\long\def\y{{\mathbf{y}}}

\global\long\def\k{{\mathbf{k}}}

\global\long\def\yt#1{{\y(#1)}}

\global\long\def\yf{{\mathbf{\yt{\cdot}}}}

\global\long\def\z{{\mathbf{z}}}

\global\long\def\v{{\mathbf{v}}}

\global\long\def\h{{\mathbf{h}}}

\global\long\def\s{{\mathbf{s}}}

\global\long\def\c{{\mathbf{c}}}

\global\long\def\p{{\mathbf{p}}}

\global\long\def\f{{\mathbf{f}}}

\global\long\def\t{{\mathbf{t}}}

\global\long\def\rb{{\mathbf{r}}}

\global\long\def\rt#1{{\rb(#1)}}

\global\long\def\rf{{\mathbf{\rt{\cdot}}}}

\global\long\def\mat#1{{\ensuremath{\bm{\mathrm{#1}}}}}

\global\long\def\matN{\ensuremath{{\bm{\mathrm{N}}}}}

\global\long\def\matX{\ensuremath{{\bm{\mathrm{X}}}}}

\global\long\def\matK{\ensuremath{{\bm{\mathrm{K}}}}}

\global\long\def\matA{\ensuremath{{\bm{\mathrm{A}}}}}

\global\long\def\matB{\ensuremath{{\bm{\mathrm{B}}}}}

\global\long\def\matC{\ensuremath{{\bm{\mathrm{C}}}}}

\global\long\def\matD{\ensuremath{{\bm{\mathrm{D}}}}}

\global\long\def\matO{\ensuremath{{\bm{\mathrm{O}}}}}

\global\long\def\matQ{\ensuremath{{\bm{\mathrm{Q}}}}}

\global\long\def\matP{\ensuremath{{\bm{\mathrm{P}}}}}

\global\long\def\matU{\ensuremath{{\bm{\mathrm{U}}}}}

\global\long\def\matV{\ensuremath{{\bm{\mathrm{V}}}}}

\global\long\def\matM{\ensuremath{{\bm{\mathrm{M}}}}}

\global\long\def\matG{\ensuremath{{\bm{\mathrm{G}}}}}

\global\long\def\calH{{\cal H}}

\global\long\def\calS{{\cal S}}

\global\long\def\calT{{\cal T}}

\global\long\def\matR{\mat R}

\global\long\def\matS{\mat S}

\global\long\def\matO{\mat O}

\global\long\def\matT{\mat T}

\global\long\def\matY{\mat Y}

\global\long\def\matI{\mat I}

\global\long\def\matJ{\mat J}

\global\long\def\matZ{\mat Z}

\global\long\def\matW{\mat W}

\global\long\def\tmatK{\widetilde{\matK}}

\global\long\def\matL{\mat L}

\global\long\def\S#1{{\mathbb{S}_{N}[#1]}}

\global\long\def\IS#1{{\mathbb{S}_{N}^{-1}[#1]}}

\global\long\def\PN{\mathbb{P}_{N}}

\global\long\def\ONorm#1{\|#1\|_{op}}

\global\long\def\ONormS#1{\|#1\|_{op}^{2}}

\global\long\def\TNormS#1{\|#1\|_{2}^{2}}

\global\long\def\TNorm#1{\|#1\|_{2}}

\global\long\def\InfNorm#1{\|#1\|_{\infty}}

\global\long\def\InfNormS#1{\|#1\|_{\infty}^{2}}

\global\long\def\FNorm#1{\|#1\|_{F}}

\global\long\def\FNormS#1{\|#1\|_{F}^{2}}

\global\long\def\UNorm#1{\|#1\|_{\matU}}

\global\long\def\NucNorm#1{\|#1\|_{\star}}

\global\long\def\UNormS#1{\|#1\|_{\matU}^{2}}

\global\long\def\UINormS#1{\|#1\|_{\matU^{-1}}^{2}}

\global\long\def\ANorm#1{\|#1\|_{\matA}}

\global\long\def\BNorm#1{\|#1\|_{\mat B}}

\global\long\def\ANormS#1{\|#1\|_{\matA}^{2}}

\global\long\def\AINormS#1{\|#1\|_{\matA^{-1}}^{2}}

\global\long\def\T{\textsc{T}}

\global\long\def\conj{\textsc{*}}

\global\long\def\pinv{\textsc{+}}

\global\long\def\Var#1{{\mathbb{V}}\left[#1\right]}

\global\long\def\Expect#1{{\mathbb{E}}\left[#1\right]}

\global\long\def\ExpectC#1#2{{\mathbb{E}}_{#1}\left[#2\right]}

\global\long\def\dotprod#1#2#3{(#1,#2)_{#3}}

\global\long\def\dotprodN#1#2{(#1,#2)_{{\cal N}}}

\global\long\def\dotprodH#1#2{(#1,#2)_{{\cal {\cal H}}}}

\global\long\def\dotprodsqr#1#2#3{(#1,#2)_{#3}^{2}}

\global\long\def\Trace#1{{\bf Tr}\left(#1\right)}

\global\long\def\nnz#1{{\bf nnz}\left(#1\right)}

\global\long\def\rank#1{{\bf rank}\left(#1\right)}

\global\long\def\vol#1{{\bf vol}\left(#1\right)}

\global\long\def\range#1{{\bf range}\left(#1\right)}

\global\long\def\sr#1{{\bf sr}\left(#1\right)}

\global\long\def\poly#1{{\bf poly}\left(#1\right)}

\global\long\def\gap#1#2{{\bf gap}_{#2}\left(#1\right)}

\global\long\def\gapS#1#2{{\bf gap}_{#2}^{2}\left(#1\right)}

\global\long\def\diag#1{{\bf diag}\left(#1\right)}

\global\long\def\Gr#1#2{{\bf Gr}(#1,#2)}

\title{Sketching for Principal Component Regression}

\author{Liron Mor-Yosef \\
 Tel Aviv University \\
 lironmo2@mail.tau.ac.il\\
 \and Haim Avron \\
 Tel Aviv University \\
haimav@post.tau.ac.il \\
}
\maketitle
\begin{abstract}
Principal component regression (PCR) is a useful method for regularizing
least squares approximations. Although conceptually simple, straightforward
implementations of PCR have high computational costs and so are inappropriate
for large scale problems. In this paper, we propose efficient algorithms
for computing approximate PCR solutions that are, on one hand, high
quality approximations to the true PCR solutions (when viewed as minimizer
of a constrained optimization problem), and on the other hand entertain
rigorous risk bounds (when viewed as statistical estimators). In particular,
we propose an input sparsity time algorithms for approximate PCR.
We also consider computing an approximate PCR in the streaming model,
and kernel PCR. Empirical results demonstrate the excellent performance
of our proposed methods.
\end{abstract}

\section{Introduction}

Least squares approximations of the form 
\[
\min_{\x\in\R^{d}}\TNorm{\matA\x-\b}
\]
are fundamental building blocks in computational science and statistical
data analysis, with applications ranging from statistical data analysis
to inverse problems. However, it is well appreciated, especially in
the aforementioned application areas, that regularization is often
the key to achieving the best results.

One of the basic methods for regularizing least squares approximations
is principal component regression (PCR)~\cite{hotelling1933analysis,kendall1957course,artemiou2009principal}.
Given a data matrix $\matA$, a right hand side $\b$ and a target
rank $k$, PCR is computed by first computing the coefficients $\matV_{\matA,k}$
corresponding to the top $k$ principal components of $\matA$ (i.e.,
to dominant right invariant subspace of $\matA$), then regressing
on $\matA\matV_{\matA,k}$ and $\b$, and finally projecting the solution
back to the original space. In short, the PCR estimator is $\x_{k}=\matV_{\matA,k}(\matA\matV_{\matA,k})^{\pinv}\b$
and regularization is achieved via PCA based dimensionality reduction.
While there is some criticism of PCR in the statistical literature~\cite{artemiou2009principal,jolliffe1982note},
it is nevertheless a valuable tool in the toolbox of practitioners. 

Up until recent breakthroughs on fast methods for least squares approximations,
there was little penalty in terms of computational complexity when
switching from ordinary least squares (OLS) to PCR. Indeed, the complexity
of SVD based computation of the dominant invariant subspace is $O(nd\min(n,d))$,
and this matches the asymptotic complexity of straightforward computation
of the OLS solution (i.e., via direct methods). However, recent progress
on fast sketching based algorithms for linear regression~\cite{DMMS11,RT08,MSM14,CW17,woodruff2014sketching}
has created a gap: exact computation of the principal components still
requires SVD so the overall complexity is still $O(nd\min(n,d))$,
even though the OLS stage is faster. The gap is not insubstantial:
when learning with large scale data (either large $n$, or large $d$),
$O(nd\min(n,d))$ is often infeasible, but modern sketching based
linear regression methods are. 

\subsection{Contributions}

In this paper, we study the use of dimensionality reduction prior
to computing PCR (so we can compute PCR on a smaller input matrix).
In particular, for a data matrix $\matA$, we relate the PCR solution
of $\matA\matR$, where $\matR$ is any dimensionality reduction matrix,
to the PCR solution of $\matA$. To do so, we study the notion of
approximate PCR both from an optimization perspective and from a statistical
perspective, and provide conditions on $\matR$ that guarantee that
after projecting the solution back to the full space (by multiplying
by $\matR^{\T}$) we have an approximate PCR solution with rigorous
statistical risk bounds. These results are described in Section~\ref{sec:dim-reduce-pcr}. 

We then leverage the aforementioned results to design fast, sketching
based, algorithms for approximate PCR. We propose algorithms specialized
for the several cases (in the following, $n$ is number of data points,
$d$ is dimension of the data): large $n$ (using left sketching),
large $d$ (using right sketching), and both $n$ and $d$ large (using
two-sided sketching). Furthermore, we propose an input-sparsity time
algorithm for approximate PCR. These results are described in Section~\ref{sec:alg}.

We also consider computing approximate PCR in the streaming model,
providing the first algorithm for computing approximate PCR in a stream.
We also provide a fast algorithm for approximate Kernel PCR (polynomial
kernel only). These results are described in Section~\ref{sec:extensions}.

Finally, empirical results (Section \ref{sec:experiments}) clearly
demonstrate the ability of our proposed algorithms to compute approximate
PCR solution, the correctness of our theoretical analysis, and the
advantages of using our techniques instead of simpler techniques like
compressed least squares. 

In general, unlike previous works on randomized methods for PCR (which
we discuss in the next subsection), we analyze the use of sketching
for PCR from a sketch-and-solve approach. We discuss the various advantages
and disadvantages of the sketch-and-solve approach in comparison to
iterative based approaches, in the next subsection.

\subsection{Related Work}

Recently matrix sketching, such as the use of random projections,
has emerged as a powerful technique for accelerating and scaling many
important statistical learning techniques. See recent surveys by Woodruff~\cite{woodruff2014sketching}
and Mahoney et al.~\cite{YMM15} for an extensive exposition on this
subject. So far, there has been limited research on the use of matrix
sketching in the context of principal component regression. 

One natural strategy for leveraging sketching in the context of PCR
is to use approximate principal components. Approximate principal
components can be computed using fast sketching based algorithm for
approximate PCA (also known as 'randomized SVD')~\cite{halko2011finding,woodruff2014sketching}.
This was recently explored by Boutsidis and Magdon-Ismail~\cite{BM14}.
The authors show that the if the number of subspace iterations is
sufficiently large, one obtain a bound on the sub-optimality of the
approximate solution and on the error of the solution vector. We too
bound the sub-optimality of our solutions, but instead of bounding
the error of the solution vector, we bound their distance to the right
dominant subspace, or bound the distance of the projection to the
left dominant subspace. 

Frostig et al. leverage fast randomized algorithms for ridge regression
to design iterative algorithms for principal component regression
and principal component projection~\cite{frostig2016principal}.
Forstig et al.'s results were later improved by Allen-Zhu and Li~\cite{AL17}.
Both of the aforementioned methods use iterations, while our work
explores the use of a sketch-and-solve approach. While it is true
that better accuracies can be achieved using iterative methods with
sketching based accelerators~\cite{RT08,AMT10,MSM14,GOS16,ACW17},
there are some advantages in using a sketch-and-solve approach. In
particular, sketch-and-solve algorithms are typically faster. However
this comes at the cost: sketch-and-solve algorithms typically provide
cruder approximations. Nevertheless, it is not uncommon for these
cruder approximations to be sufficient in machine learning applications.
Another advantage of the sketch-and-solve approach is that it is more
amenable to streaming and kernelization; we consider both in this
paper. 

Closely related to our work is recent work on Compressed Least Squares
(CLS)~\cite{maillard2009compressed,Kaban14,slawski2017compressed,Slawski17,THM17}.
In particular, our statistical analysis (section~\ref{subsec:statistical})
is inspired by recent statistical analysis of CLS~\cite{slawski2017compressed,Slawski17,Kaban14}.
Additionally, CLS is sometimes considered as a computationally attractive
alternative to PCR~\cite{Slawski17,THM17}. While CLS certainly uses
matrix sketching to compress the matrix, it also uses the compression
to regularize the problem. The mix between compression for scalability
and compression for regularization reduces the ability to fine tune
the method to the needs at hand, and thus obtain the best possible
results. In contrast, our methods uses sketching primarily to approximate
the principal components and as such serves as a means for scalability
only. We propose methods that are computationally as attractive as
CLS, and are more faithful to the behavior of PCR (in fact, CLS is
a special case of one of our proposed algorithms). These advantages
over CLS are also evident in the experimental results reported in
Section~\ref{sec:experiments}. 

Principal component regression is a form of least squares regression
with convex constraints (once the dominant subspace has been found).
Pilanchi and Wainwright recently explored the effect of regularization
on the sketch size for least squares regression~\cite{pilanci2015randomized,pilanci2016iterative}.
In the aforementioned papers, sketching is applied only to the objective,
while the constraint is enforced exactly. This is unsatisfactory in
the context of PCR since for PCR the constraints are gleaned from
the input, and enforcing them is as expensive as solving the problem
exactly. In contrast, our method uses sketching not only to compress
the objective function, but also to approximate the constraint set. 

Ridge regression (also known as Tikhonov regularization) is another
popular and well studied method for regularizing least squares solutions.
It also closely related to PCR in the sense that the ridge term can
be viewed as a soft damping of the singular values. Recently several
sketching-based algorithms have been suggested to accelerate the solution
of ridge regression~\cite{ChenEtAl15,ACW17b,WGM17,CYD18}.

\section{Preliminaries}

\subsection{Notation and Basic Definitions}

We denote scalars using Greek letters or using $x,y,\dots$. Vectors
are denoted by $\x,\y,\dots$ and matrices by $\matA,\mat B,\dots$.
The $s\times s$ identity matrix is denoted $\matI_{s}$. We use the
convention that vectors are column-vectors. $\nnz{\matA}$ denotes
the number of non-zeros in $\matA$. The notation $\alpha=(1\pm\gamma)\beta$
means that $(1-\gamma)\beta\leq\alpha\leq(1+\gamma)\beta$, and the
notation $\alpha=\beta\pm\gamma$ means that $|\alpha-\beta|\leq\gamma$. 

Given a matrix $\matX\in\R^{m\times n}$, let $\matX=\matU_{\matX}\Sigma_{\matX}\matV_{\matX}^{\T}$
be a \emph{thin SVD} of $\matX$, i.e. $\matU_{\matX}\in\R^{m\times\min(m,n)}$
is a matrix with orthonormal columns, $\Sigma_{\matX}\in\R^{\min(m,n)\times\min(m,n)}$
is a diagonal matrix with the non-negative singular values on the
diagonal, and $\matV_{\matX}\in\R^{n\times\min(m,n)}$ is a matrix
with orthonormal columns. The thin SVD decomposition is not necessarily
unique, so when we use this notation we mean that the statement is
correct for any such decomposition. A thin SVD decomposition can be
computed in $O(mn\min(m,n))$. We denote the singular values of $\matX$
by $\sigma_{\max}(\matX)=\sigma_{1}(\matX)\geq\dots\geq\sigma_{\min(m,n)}(\matX)=\sigma_{\min}(\matX)$,
omitting the matrix from the notation if the relevant matrix is clear
from the context. For $k\leq\min(m,n)$, we use $\matU_{\matX,k}$
(respectively $\matV_{\matX,k})$ to denote the matrix consisting
of the first $k$ columns of $\matU_{\matX}$ (respectively $\matV_{\matX}$),
and use $\Sigma_{\matX,k}$ to denote the leading $k\times k$ minor
of $\Sigma_{\matX}$. We use $\matU_{\matX,k+}$ (respectively $\matV_{\matX,k+})$
to denote the matrix consisting of the last $\min(m,n)-k$ columns
of $\matU_{\matX}$ (respectively $\matV_{\matX}$), and use $\Sigma_{\matX,k+}$
to denote the lower-right $(\min(m,n)-k)\times(\min(m,n)-k)$ block
of $\Sigma_{\matX}$. In other words, 
\[
\matU_{\matX}=\left[\begin{array}{cc}
\matU_{\matX,k} & \matU_{\matX,k+}\end{array}\right]\quad\Sigma_{\matX}=\left[\begin{array}{cc}
\Sigma_{\matX,k} & 0\\
0 & \Sigma_{\matX,k+}
\end{array}\right]\quad\matV_{\matX}=\left[\begin{array}{cc}
\matV_{\matX,k} & \matV_{\matX,k+}\end{array}\right].
\]
The \emph{Moore-Penrose pseudo-inverse} of $\matX$ is $\matX^{\pinv}\coloneqq\matV_{\matX}\Sigma_{\matX}^{\pinv}\matU_{\matX}^{\T}$
where $\Sigma_{\matX}^{\pinv}=\diag{\sigma_{1}(\matX)^{\pinv},\dots,\sigma_{\min(m,n)}(\matX)^{\pinv}}$
with $a^{+}=a^{-1}$ when $a\neq0$ and $0$ otherwise. 

The \emph{stable rank }of a matrix $\matX$ is $\sr{\matX}\coloneqq\FNormS{\matX}/\TNormS{\matX}$.
The $k$-th \emph{relative gap} of a matrix $\matX$ is 
\[
\gap{\matX}k=\frac{\sigma_{k}^{2}-\sigma_{k+1}^{2}}{\sigma_{1}^{2}}\,.
\]

For a subspace ${\cal U}$, we use $\matP_{{\cal U}}$ to denote the
\emph{orthogonal projection matrix} onto ${\cal U}$, and $\matP_{\matX}$
for the projection matrix on the column space of $\matX$ (i.e. $\matP_{\matX}=\matP_{\range{\matX}}$).
We have $\matP_{\matX}=\matX\matX^{\pinv}$. The \emph{complementary
projection matrix }is $\matP_{\matX}^{\perp}=\matI-\matP_{\matX}$.
A useful property of projection matrices is that if ${\cal S}\subseteq\calT$
then $\matP_{\calS}\matP_{\calT}=\matP_{\calT}\matP_{{\cal S}}=\matP_{\calS}$.
Furthermore, we note the following result.
\begin{thm}
[Theorem 2.3 in \cite{Stewart_perturbation}]\label{thm:projection_equality}
For any $\matA$ and $\matB$ with the same number of rows, the following
statements hold:
\begin{enumerate}
\item If $\rank{\matA}=\rank{\matB}$, then the singular values of $\matP_{\matA}\matP_{\matB}^{\perp}$
and $\matP_{\matB}\matP_{\matA}^{\perp}$ are the same, so
\[
\TNorm{\matP_{\matA}\matP_{\matB}^{\perp}}=\TNorm{\matP_{\matB}\matP_{\matA}^{\perp}}
\]
\item Moreover the nonzero singular values $\sigma$ of $\matP_{\matA}\matP_{\matB}^{\perp}$
correspond to pairs $\pm\sigma$ of eigenvalues of $\matP_{\matB}-\matP_{\matA}$,
so 
\[
\TNorm{\matP_{\matB}-\matP_{\matA}}=\TNorm{\matP_{\matA}\matP_{\matB}^{\perp}}
\]
\item If $\TNorm{\matP_{\matB}-\matP_{\matA}}<1$, then $\rank{\matA}=\rank{\matB}$.
\end{enumerate}
\end{thm}

\subsection{Principal Component Regression and Principal Component Projection }

In the \emph{Principal Component Regression (PCR)} problem, we are
given an input $n\textrm{-by-}d$ data matrix $\matA$, a right hand
side $\b\in\R^{n}$, and a rank parameter $k$ which is smaller or
equal to the rank of $\matA$. Furthermore, we assume that there is
an non-zero eigengap at $k$: $\sigma_{k}>\sigma_{k+1}$. The goal
is to find the PCR solution, $\x_{k}$, defined as 
\begin{equation}
\x_{k}\coloneqq\arg\min_{\x\in\range{\matV_{\matA,k}}}\TNorm{\matA\x-\b}.\label{eq:pcr-opt-problem}
\end{equation}
It is easy to verify that $\x_{k}=\matV_{\matA,k}(\matA\matV_{\matA,k})^{\pinv}\b=\matV_{\matA,k}\Sigma_{\matA,k}^{-1}\matU_{\matA,k}^{\T}\b$.
The \emph{Principal Component Projection (PCP) }of $\b$ is $\b_{k}\coloneqq\matA\x_{k}=\matP_{\matU_{\matA,k}}\b$. 

Straightforward computation of $\x_{k}$ and $\b_{k}$ via the SVD
takes $O(nd\min(n,d))$ operations\footnote{The complexity when using iterative algorithms (e.g. Lanczos) to compute
only the dominant invariant spaces depend on several additional facts
and in particular on spectral properties of the matrix and sparsity
level. Thus, to avoid overly complicating the discussion on computational
complexity, we refrain from further discussion of iterative methods
for computing dominant eigenspaces}. We are primarily interested in finding faster algorithms that compute
an approximate PCR or PCP solution (we formalize the terms 'approximate
PCP/PCR' in Section~\ref{sec:dim-reduce-pcr}). Throughout the paper,
we use $\matA,\b,$ and $k$ as the arguments of the PCR/PCP problem
to be solved.

\subsection{\label{subsec:Matrix-Perturbations}Matrix Perturbations and Distance
Between Subspaces}

Our analysis uses matrix perturbation theory extensively. We now describe
the basics of this theory and the results we use. 

The \emph{principal angles} $\theta_{j}\in[0,\pi/2]$ between two
subspaces ${\cal U}$ and ${\cal W}$ are recursively defined by the
identity
\[
\cos(\theta_{j})=\max_{\u\in{\cal U}}\max_{\w\in{\cal W}}\u^{\T}\w\,\text{s.t.}\,\TNorm{\u}=1,\TNorm{\w}=1,\forall i<j.\u_{i}^{\T}\u=0,\w_{i}^{\T}\w=0\,.
\]
We use $\u_{j}$ and $\w_{j}$ to denote the vectors for which $\cos(\theta_{j})=\u_{j}^{\T}\w_{j}$.
Let $\Theta({\cal {\cal U}},{\cal W})$ denote the $d\times d$ diagonal
matrix whose $j$th diagonal entry is the $j$th principal angle,
and as usual we allow writing matrices instead of subspaces as short-hand
for the column space of the matrix. Henceforth, when we write a function
on $\Theta(\cdot,\cdot)$, i.e. $\sin(\Theta(\matU,\text{\ensuremath{\matW}}$)),
we mean evaluating the function entrywise on the diagonal only. It
is well known~ \cite[section 6.4.3]{golub2012matrix} that if $\matU$
(respectively $\matW$) is a matrix with orthonormal columns whose
column space is equal to ${\cal U}$ (respectively ${\cal W}$) then
\[
\sigma_{j}(\matU^{\T}\matW)=\cos(\theta_{j}).
\]
The following lemma connects the tangent of the principal angles to
the spectral norm of an appropriate matrix.
\begin{lem}
[Lemma 4.3 in \cite{drineas2016structural}] \label{lem:tan-to-spectral}Let
$\matQ\in\R^{n\times s}$ have orthonormal columns, and let $\matW=(\begin{array}{cc}
\matW_{k} & \matW_{k+}\end{array})\in\R^{n\times n}$ be an orthogonal matrix where $\matW_{k}\in\R^{n\times k}$ with
$k\leq s$. If $\rank{\matW_{k}^{\T}\matQ}=k$ then 
\[
\TNorm{\tan\Theta(\matQ,\matW_{k})}=\TNorm{(\matW_{k+}^{\T}\matQ)(\matW_{k}^{\T}\matQ)}\,.
\]
\end{lem}

Matrix perturbation theory studies how a perturbation of a matrix
translate to perturbations of the matrix's eignevalues and eigenspaces.
In order to bound the perturbation of an eigenspace, one needs some
notion of distance between two subspaces. One common distance metric
between two subspaces is 
\begin{equation}
d_{2}({\cal U},{\cal W})\coloneqq\TNorm{\matP_{{\cal U}}-\matP_{{\cal W}}}\,.\label{eq:distance_between_subspaces}
\end{equation}
If $\matU$ and $\matV$ have the same number of columns, and both
have orthonormal columns, then 
\[
d_{2}(\matU,\matV)=\sqrt{1-\sigma_{\min}(\matU^{\T}\matV)^{2}}=\sin(\theta_{\max})=\TNorm{\sin\Theta(\matU,\matV)}
\]
where $\theta_{\max}$ is the maximum principal angle between $\range{\matU}$
and $\range{\matV}$~\cite[section 6.4.3]{golub2012matrix}.

A classical result that bounds the distance between the dominant subspaces
of two symmetric matrices in terms of the spectral norm of difference
between the two matrices is the Davis-Kahan $\sin(\Theta)$ theorem~\cite[Section 2]{davis1970rotation}.
We need the following corollary of this theorem:
\begin{thm}
[Corollary of Davis-Kahan $\sin\Theta$ Theorem \cite{davis1970rotation}]\label{thm:sin-theta-1}
Let $\matA,\tilde{\matA}\in\R^{n\times n}$ be two symmetric matrices,
both of rank at least $k$. Suppose that $\lambda_{k}>\tilde{\lambda}_{k+1}$
where $\lambda_{1}\geq\dots\geq\lambda_{n}$ and $\tilde{\lambda}_{1}\geq\dots\geq\tilde{\lambda}_{n}$
are the eigenvalues of $\matA$ and $\tilde{\matA}$. We have 
\[
d_{2}(\matV_{\matA,k},\matV_{\tilde{\matA},k})\leq\frac{\TNorm{\matA-\tilde{\matA}}}{\lambda_{k}-\tilde{\lambda}_{k+1}}\,.
\]
\end{thm}

\begin{proof}
We use the following variant of the $\sin\Theta$ Theorem (see~\cite[Theorem 2.16]{stewart2001matrix}):
suppose a symmetric matrix $\matB$ has a spectral representation
\[
\matB=\matX\matL\matX^{\T}+\matY\matM\matY^{\T}
\]
where $\left[\matX\,\matY\right]$ is square orthonormal. Let the
orthonormal matrix $\matZ$ be of the same dimensions as $\matX$
and suppose that 
\[
\matR=\matB\matZ-\matZ\matN
\]
where $\matN$ is symmetric. Furthermore, suppose that the spectrum
of $\matN$ is contained in some interval $[\alpha,\beta]$ and that
for some $\delta>0$ the spectrum of $\matM$ lies outside of $[\alpha-\delta,\beta+\delta]$.
Then, 
\[
\TNorm{\sin\Theta(\matX,\matZ)}\leq\frac{\TNorm{\matR}}{\delta}\,.
\]
We prove Theorem~\ref{thm:sin-theta-1} by applying the aforementioned
variant of the $\sin\Theta$ Theorem with: $\matB=\tilde{\matA}$,
$\matX=\matV_{\tilde{\matA},k}$, $\matY=\matV_{\tilde{\matA},k+}$,
$\matL=\diag{\tilde{\lambda}_{1},\dots,\tilde{\lambda}_{k}}$, $\matM=\diag{\tilde{\lambda}_{k+1},\dots,\tilde{\lambda}_{n}},$
$\matZ=\matV_{\matA,k}$, $\matN=\diag{\lambda_{1},\dots,\lambda_{k}}$,
and $\delta=\lambda_{k}-\tilde{\lambda}_{k+1}$ . It is easy to verify
that the conditions of the $\sin\Theta$ Theorem hold, so 
\[
\TNorm{\sin\Theta(\matV_{\matA,k},\matV_{\tilde{\matA},k})}\leq\frac{\TNorm{\matR}}{\lambda_{k}-\tilde{\lambda}_{k+1}}
\]
where $\matR=\tilde{\matA}\matV_{\matA,k}-\matV_{\matA,k}\matN$.
We have $\matA\matV_{\matA,k}=\matV_{\matA,k}\matN$ so $\TNorm{\matR}=\TNorm{(\tilde{\matA}-\matA)\matV_{\matA,k}}\leq\TNorm{\tilde{\matA}-\matA}$.
Combining this inequality with the previous one and noting that $d_{2}(\matV_{\matA,k},\matV_{\tilde{\matA},k})=\TNorm{\sin\Theta(\matV_{\matA,k},\matV_{\tilde{\matA},k})}$
completes the proof.
\end{proof}
Under the conditions of Theorem~\ref{thm:sin-theta-1}, since $\matA$
and $\tilde{\matA}$ are symmetric matrices, Weyl's inequality implies
that 
\[
d_{2}(\matV_{\matA,k},\matV_{\tilde{\matA},k})\leq\frac{\TNorm{\matA-\tilde{\matA}}}{\lambda_{k}-\lambda_{k+1}-\TNorm{\matA-\tilde{\matA}}}\,
\]
as long as $\TNorm{\matA-\tilde{\matA}}<\lambda_{k}-\lambda_{k+1}$.
Thus, if $\TNorm{\matA-\tilde{\matA}}\ll\lambda_{k}-\lambda_{k+1}$
then we can compute an approximation to the $k$-dimensional dominant
subspace of $\matA$ by computing the $k$-dimensional dominant subspace
of $\tilde{\matA}$.

\section{\label{sec:dim-reduce-pcr}PCR with Dimensionality Reduction}

Our goal is to design algorithms which compute an approximate solution
to the PCR or PCP problem. Our strategy for designing such algorithms
is to reduce the dimensions of $\matA$ prior to computing the PCR/PCP
solution. Specifically, let $\matR\in\R^{d\times t}$ be some matrix
where $t\leq d$, and define 
\begin{equation}
\x_{\matR,k}:=\matR\matV_{\matA\matR,k}(\matA\matR\matV_{\matA\matR,k})^{+}\b\,.\label{eq:apcr}
\end{equation}
The rationale in Eq.~(\ref{eq:apcr}) is as follows. First, $\matA$
is compressed by computing $\matA\matR$ (this is the dimensionality
reduction step). Then we compute the rank $k$ PCR solution of $\matA\matR$
and $\b$; this is $(\matA\matR\matV_{\matA\matR,k})^{+}\b$. Finally,
the solution is projected back to the original space by multiplying
by $\matR\matV_{\matA\matR,k}$. Obviously, given $\matR$ we can
compute $\x_{\matR,k}$ in $O(ndt)$ (and even faster, if $\matA$
is sparse), so if $t\ll\min(n,d)$ there is a potential for significant
gain in terms of computational complexity\emph{ }provided it is possible
to compute $\matR$ efficiently as well. Furthermore, if we design
$\matR$ to have some special structure that allows us to compute
$\matA\matR$ in $O(nt^{2})$ time, the overall complexity would reduce
to $O(nt^{2})$.

Of course, $\x_{\matR,k}$ is not the PCR solution $\x_{k}$ (unless
$\matR=\matV_{\matA,k}$). This suggests the following mathematical
question (which, in turn, leads to an algorithmic question): under
which conditions on $\matR$ is $\x_{\matR,k}$ a good approximation
to the PCR solution $\x_{k}$? In this section, we derive general
conditions on $\matR$ that ensure deterministically that $\x_{\matR,k}$
is in some sense (which we formalize later in this section) a good
approximation of $\x_{k}$. The results in this section are non algorithmic
and independent of the method in which $\matR$ is computed. In the
next section we address the algorithmic question: how can we compute
such $\matR$ matrices efficiently? 

We approach the mathematical question from two different perspectives:
an optimization perspective and a statistical perspective. In the
optimization perspective, we consider PCR/PCP as an optimization problem
(Eq.~(\ref{eq:pcr-opt-problem})), and ask whether the value of the
objective function of $\x_{\matR,k}$ is close to optimal value of
the objective function, while upholding the constraints approximately
(see Definition~\ref{def:apcr-apcp}). In the statistical perspective,
we treat $\x_{k}$ and $\x_{\matR,k}$ as statistical estimators,
and compare their excess risk under a fixed-design model. Interestingly,
the conditions we derive for $\matR$ are the same for both perspectives. 

Before proceeding, we remark that an important special case of (\ref{eq:apcr})
is when $\matR$ has exactly $k$ columns. In that case, for brevity,
we omit the subscript $k$ from $\x_{\matR,k}$ and notice that 
\begin{equation}
\x_{\matR}=\matR(\matA\matR)^{+}\b\,.\label{eq:cls}
\end{equation}
Eq.~(\ref{eq:cls}) is valid even if $\matR$ has more than $k$
columns and/or the columns are not orthonormal. Thus, an established
technique in the literature, frequently referred to as Compressed
Least Squares (CLS)~\cite{maillard2009compressed,Kaban14,slawski2017compressed,Slawski17,THM17},
is to generate a random $\matR$ and compute $\x_{\matR}$. To avoid
confusion, we stress the difference between (\ref{eq:apcr}) and (\ref{eq:cls}):
in (\ref{eq:apcr}) we compute a PCR solution on the compressed matrix
$\matA\matR$, while in~(\ref{eq:cls}) ordinary least squares is
used. These two strategies coincide when $\matR$ has $k$ columns.
In this paper, we focus on Eq.~(\ref{eq:apcr}) and consider Eq.~(\ref{eq:cls})
only when it is a special case of Eq.~(\ref{eq:apcr}) (when $\matR$
has exactly $k$ columns). For an analysis of CLS from a statistical
perspective, see recent work by Slawski~\cite{Slawski17}. 

\subsection{Optimization Perspective}

The PCR solution can be written as the solution of a constrained least
squares problem:
\[
\x_{k}=\arg\min_{\begin{array}{c}
\TNorm{\matV_{\matA,k+}^{\T}\x}=0\\
\x\in\range{\matA^{\T}}
\end{array}}\TNorm{\matA\x-\b}.
\]
In order to analyze a candidate solution $\tilde{\x}$ from an optimization
perspective, we need to decide how to treat the constraints. One option
is to require a candidate $\tilde{\x}$ to be inside the feasible
set. Indeed, Pilanci and Wainwright recently considered sketching
based methods for constrained least squares regression~\cite{pilanci2015randomized}.
However, there is no evident way to impose $\matV_{\matA,k+}^{\matT}\x=0$
without actually computing $\matV_{\matA,k+}$, which is as expensive
as computing $\matV_{\matA,k}$. Thus, if we require an approximate
solution to be inside the feasible set, we might as well compute the
exact PCR solution. Thus, in our notion of approximate PCR, we relax
the constraints and require only that the approximate solution is
close to meeting the constraint, i.e. we seek a solution for which
$\TNorm{\matA\tilde{\x}-\b}$ is close to $\TNorm{\matA\x_{k}-\b}$
\emph{and} $\TNorm{\matV_{\matA,k+}^{\T}\tilde{\x}}$ is small. 

Similarly, the if $\matA$ has full rank the PCP solution can written
as the solution of a constrained least squares problem:
\[
\b_{k}=\arg\min_{\begin{array}{c}
\TNorm{\matU_{\matA,k+}^{\T}\tilde{\b}}=0\\
\tilde{\b}\in\range{\matA}
\end{array}}\TNorm{\tilde{\b}-\b}
\]
Again, our notion of approximate PCP relaxes the constraint.

The discussion above motivates the following definition of approximate
PCR/PCP. 
\begin{defn}
[Approximate PCR and PCP] \label{def:apcr-apcp}An estimator $\tilde{\x}$
is an \emph{$(\epsilon,\upsilon)$-approximate PCR} of rank $k$ if
\[
\TNorm{\matA\tilde{\x}-\b}=\TNorm{\matA\x_{k}-\b}\pm\epsilon\TNorm{\b}
\]
and $\TNorm{\matV_{\matA,k+}^{\T}\tilde{\x}}\leq\upsilon\TNorm{\b}$.
An estimator $\tilde{\b}$ is an \emph{($\epsilon,\upsilon)$-approximate
PCP} of rank $k$ if 
\[
\TNorm{\tilde{\b}-\b}=\TNorm{\b_{k}-\b}\pm\epsilon\TNorm{\b}
\]
and $\TNorm{\matU_{\matA,k+}^{\T}\tilde{\b}}\leq\upsilon\TNorm{\b}$.
\end{defn}

Before proceeding, a few remarks are in order.
\begin{enumerate}
\item Imposing no constraints on $\tilde{\x}$ (or $\tilde{\b}$) does not
make sense: we can always form or approximate the ordinary least squares
solution and it will demonstrate a smaller objective value. Indeed,
the main motivation for using PCR to impose some form of regularization,
so it is crucial the definition of approximate PCR/PCP have some form
of regularization built-in.
\item We require only additive error on the objective function, while relative
error bounds are usually viewed as more desirable. For approximate
PCR, requiring relative error bounds is likely unrealistic: since
it is possible that $\b=\matA\x_{k}$, any algorithm that provides
a relative error bound must search inside a space that contains $\range{\matV_{\matA,k}}$.
This is a strong restriction (and plausibly one that actually requires
computing $\matV_{\matA,k}$ ).
\item Approximate PCR implies approximate PCP: if $\tilde{\x}$ is an $(\epsilon,\nu)$-approximate
PCR then $\matA\tilde{\x}$ is an $(\epsilon,\sigma_{k+1}\nu)$-approximate
PCP. 
\item Our notion of approximate PCP is somewhat similar to the notion of
approximate PCP proposed recently by Allen-Zhu and Li~\cite{AL17}.
\item Yet another notion of approximate PCR appears in~\cite[Theorem 5]{BM14}.
They too, consider an additive error on objective function, but instead
of considering the distance to the dominant subspace they bound the
distance of the approximate solution to the true solution. We remark
that a bound on $\TNorm{\x_{k}-\tilde{\x}}$ trivially implies a bound
on $\TNorm{\matV_{\matA,k+}^{\T}\tilde{\x}}$.
\item Arguably, it would have been preferable to require the approximate
PCR solution $\tilde{\x}$ to be such that $\TNorm{\x_{k}-\tilde{\x}}$
is small (relative to $\TNorm{\x_{k}}$). However, be believe that
providing such guarantees with reasonable sketch sizes requires iterations.
In this paper, we focus predominately on algorithms that do not require
iterations (the only exception being the input sparsity algorithm
in subsection~\ref{subsec:input-sparsity}). 
\end{enumerate}
We are now ready to state general conditions on $\matR$ that ensure
deterministically that $\x_{\matR}$ is an approximate PCR, and conditions
on $\matR$ that ensure deterministically that $\matA\x_{\matR,k}$
is an approximate PCP. 

\begin{thm}
\label{thm:structural}Suppose that $\matR\in\R^{d\times s}$ where
$s\ge k$. Assume that $\nu\in(0,1)$. 
\begin{enumerate}
\item If $d_{2}\left(\matU_{\mat A\matR,k},\matU_{\mat A,k}\right)\leq\nu$
then $\matA\x_{\matR,k}$ is an $(\nu,\nu)$-approximate PCP. 
\item \label{enu:structural_k_equal_s}If $s=k$, $\matR$ has orthonormal
columns (i.e., $\matR^{\T}\matR=\matI_{k})$ and $d_{2}\left(\matR,\matV_{\mat A,k}\right)\leq\nu(1+\nu^{2})^{-1/2}$
then $\x_{\matR}$ is an $\left(\frac{\sigma_{k+1}}{\sigma_{k}}\nu,\frac{\nu}{\left(\sqrt{1-\nu^{2}}-\nu\right)\sigma_{k}}\right)$-approximate
PCR. 
\end{enumerate}
\end{thm}

Before proving this theorem, we state a theorem which is a corollary
of a more general result proved recently by Drineas et al.~\cite{drineas2016structural},
and then proceed to proving a couple of auxiliary lemmas. 
\begin{thm}
[Corollary of Theorem 2.1 in \cite{drineas2016structural}]\label{thm:krylov_srtuctural}
Let $\matA$ be an $m\times n$ matrix with singular value decomposition
$\matA=\matU_{\matA}\mat{\Sigma}_{\matA}\matV_{\matA}^{\matT}$ .
Let $k\ge0$ and let $\matR\in\R^{d\times k}$ be any matrix such
that $\matV_{\matA,k}^{\matT}\matR$ has full rank. Then,
\end{thm}

\[
\TNorm{\sin\Theta(\matA\matR,\matU_{\matA,k})}\leq\TNorm{\mat{\Sigma}_{\matA,k+}}\cdot\TNorm{\mat{\Sigma}_{\matA,k}^{-1}}\cdot\TNorm{\tan\Theta(\matR,\matV_{\matA,k})}
\]
\begin{lem}
\label{lem:11}Assume $\rank{\matA}\ge k$. If $\matR\in\R^{d\times k}$
has orthonormal columns and $d_{2}\left(\matR,\matV_{\mat A,k}\right)\leq\nu$
then the following bounds hold:

\begin{equation}
\TNorm{\matV_{\mat A,k+}^{\T}\matR}\leq\nu\label{eq:orho-to-minor}
\end{equation}
\begin{equation}
\sigma_{\min}\left(\matA\matR\right)\ge\sigma_{k}\left(\sqrt{1-\nu^{2}}-\nu\right)\label{eq:singular_values_bound}
\end{equation}
Furthermore, if $\nu<1$ then $\rank{\matA\matR}=k$.
\end{lem}

\begin{proof}
Since both $\matV_{\matA,k}$ and $\matR$ have orthonormal columns,
$d_{2}\left(\matR,\matV_{\mat A,k}\right)\leq\nu$ implies that the
square of the singular values of $\matV_{\mat A,k}^{\T}\matR$ lie
inside the interval $[1-\nu^{2},1]$. The eigenvalues of $\matR^{\T}\matV_{\mat A,k}\matV_{\mat A,k}^{\T}\matR$
are exactly the square of the singular values of $\matV_{\mat A,k}^{\T}\matR$,
so the eigenvalues of $\matI_{k}-\matR^{\T}\matV_{\mat A,k}\matV_{\mat A,k}^{\T}\matR$
lie in $[0,\nu^{2}]$. Let $\matZ$ be any matrix with orthonormal
columns that completes $\matV_{\matA}$ to a basis (i.e. $\matV_{\matA}\matV_{\matA}^{\T}+\matZ\matZ^{\T}=\matI_{d}$)
and is orthogonal to $\matV_{\matA}$ (i.e., $\matV_{\matA}^{\T}\matZ=0$).
Note that $\matZ$ can be an empty matrix if $d\leq n$. Denote $\matV_{\matA,k\perp}=\left[\begin{array}{cc}
\matV_{\matA,k+} & \matZ\end{array}\right]$. We have 
\begin{eqnarray*}
\TNorm{\matV_{\mat A,k\perp}^{\T}\matR} & = & \sqrt{\TNorm{\matR^{\T}\matV_{\mat A,k\perp}\matV_{\mat A,k\perp}^{\T}\matR}}\\
 & = & \sqrt{\TNorm{\matI_{k}-\matR^{\T}\matV_{\mat A,k}\matV_{\mat A,k}^{\T}\matR}}\\
 & \leq & \nu
\end{eqnarray*}
where we used the fact that $\matV_{\mat A,k}\matV_{\mat A,k}^{\T}+\matV_{\mat A,k\perp}\matV_{\mat A,k\perp}^{\T}=\matI_{d}$.
We now note that $\matV_{\matA,k+}^{\T}\matR$ is a submatrix of $\matV_{\matA,k\perp}^{\T}$
so $\TNorm{\matV_{\mat A,k+}^{\T}\matR}\leq\TNorm{\matV_{\mat A,k\perp}^{\T}\matR}\leq\nu$.
This establishes the first part of the theorem.

As for the second part, recall the following identities: 1) for any
matrix $\matX$ and $\matY$ of the same size: $\sigma_{\min}\left(\matX\pm\matY\right)\ge\sigma_{\min}\left(\matX\right)-\sigma_{\max}\left(\matY\right)$
\cite[Theorem 3.3.19]{horn1990matrix}, 2) if the number of rows in
$\matX$ and $\matY$ is at least as large as the number of columns,
and $\matX\matY$ is defined, then $\sigma_{\min}\left(\matX\matY\right)\ge\sigma_{\min}\left(\matX\right)\sigma_{\min}\left(\matY\right)$.
We have

\begin{eqnarray*}
\sigma_{\min}\left(\matA\matR\right) & = & \sigma_{\min}(\matA\matV_{\mat A,k}\matV_{\mat A,k}^{\T}\matR+\matA\matV_{\mat A,k+}\matV_{\mat A,k+}^{\T}\matR)\\
 & \geq & \sigma_{\min}(\matA\matV_{\mat A,k}\matV_{\mat A,k}^{\T}\matR)-\sigma_{\max}(\matA\matV_{\mat A,k+}\matV_{\mat A,k+}^{\T}\matR)\\
 & \geq & \sigma_{\min}(\matA\matV_{\mat A,k})\sigma_{\min}(\matV_{\mat A,k}^{\T}\matR)-\sigma_{\max}(\matA\matV_{\mat A,k+})\sigma_{\max}(\matV_{\mat A,k+}^{\T}\matR)\\
 & = & \sigma_{k}\sigma_{\min}(\matV_{\mat A,k}^{\T}\matR)-\sigma_{k+1}\sigma_{\max}(\matV_{\mat A,k+}^{\T}\matR)\\
 & \geq & \sigma_{k}\sqrt{1-\nu^{2}}-\sigma_{k+1}\nu\\
 & \geq & \sigma_{k}(\sqrt{1-\nu^{2}}-\nu)
\end{eqnarray*}
where the first equality follows from the fact that $\matA(\matV_{\mat A,k}\matV_{\mat A,k}^{\T}+\matV_{\mat A,k+}\matV_{\mat A,k+}^{\T})=\matA$.
When $\nu<1$ we have $\sigma_{\min}(\matA\matR)>0$, so indeed the
rank of $\matA\matR$ is $k$.
\end{proof}
\begin{lem}
\label{lem:14}Assume $\rank{\matA}\ge k$. Suppose that $\matR\in\R^{d\times k}$
has orthonormal columns, and that $d_{2}\left(\matR,\matV_{\mat A,k}\right)\leq\nu(1+\nu^{2})^{-1/2}<1$.
We have 
\[
d_{2}(\matU_{\matA\matR},\matU_{\matA,k})\leq\frac{\sigma_{k+1}}{\sigma_{k}}\nu\,.
\]
\end{lem}

\begin{proof}
Since $\rank{\matA}\ge k$ and $\nu(1+\nu^{2})^{-1/2}<1$, according
to Lemma \ref{lem:11} the matrix $\matA\matR$ has full rank. According
to Theorem~ \ref{thm:projection_equality} and the fact that $\matP_{\matA\matR}$
and $\matP_{\matU_{\matA,k}}$ are orthogonal projections we have
\begin{equation}
d_{2}\left(\matU_{\matA\matR},\matU_{\matA,k}\right)=d_{2}\left(\matA\matR,\matU_{\matA,k}\right)=\TNorm{\matP_{\matA\matR}-\matP_{\matU_{\matA,k}}}=\TNorm{\matP_{\matA\matR}^{\perp}\matP_{\matU_{\matA,k}}}\label{eq:projection_equality_k0}
\end{equation}
Combining Theorem \ref{thm:krylov_srtuctural} and Eq. (\ref{eq:projection_equality_k0}),
we bound:

\begin{eqnarray*}
d_{2}\left(\matA\matR,\matU_{\matA,k}\right) & = & \TNorm{\matP_{\matA\matR}^{\perp}\matP_{\matU_{\matA,k}}}\\
 & = & \TNorm{(\matI-\matP_{\matA\matR})\matU_{\matA,k}}\\
 & = & \TNorm{\sin\Theta(\matA\matR,\matU_{\matA,k})}\\
 & \leq & \TNorm{\mat{\Sigma}_{\matA.k+}}\cdot\TNorm{\mat{\Sigma}_{\matA.k}^{-1}}\cdot\TNorm{\tan\Theta(\matR,\matV_{\matA,k})}\\
 & = & \frac{\sigma_{k+1}}{\sigma_{k}}\cdot\TNorm{\tan\Theta(\matR,\matV_{\matA,k})}\\
 & \leq & \frac{\sigma_{k+1}}{\sigma_{k}}\nu
\end{eqnarray*}
where the last inequality follows from the fact that $\Theta(\matR,\matV_{\matA,k})$
is a diagonal matrix whose diagonal values are the inverse cosine
of the singular values of $\matR^{\T}\matV_{\matA,k}$, and these,
in turn, are all larger than $\sqrt{1-\nu^{2}(1+\nu^{2})^{-1}}.$ 
\end{proof}
We are now ready to prove Theorem~\ref{thm:structural}.
\begin{proof}
[Proof of Theorem~\ref{thm:structural}]We need to show both the additive
error bounds on the objective function, and the error bound on the
constraints. We start with the additive error bounds on the objective
function, both for PCP (first part of the theorem) and for PCR (second
part of the theorem). We have 
\[
\matA\x_{\matR,k}=\matA\matR\matV_{\matA\matR,k}\left(\matA\matR\matV_{\matA\matR,k}\right)^{+}\b=\matP_{\matU_{\mat A\matR,k}}\b
\]
and 
\[
\matA\x_{k}=\mat A\matV_{\matA,k}(\matA\matV_{\matA,k})^{+}\b=\matP_{\matU_{\mat A,k}}\b\,.
\]
Thus, 
\begin{eqnarray*}
\TNorm{\matA\x_{\matR,k}-\b} & = & \TNorm{\matA\x_{k}-\b+\matA\x_{\matR,k}-\matA\x_{k}}\\
 & = & \TNorm{\matA\x_{k}-\b}\pm\TNorm{\matA\x_{\matR,k}-\matA\x_{k}}\\
 & = & \TNorm{\matA\x_{k}-\b}\pm\TNorm{(\matP_{\matU_{\mat A\matR,k}}-\matP_{\matU_{\matA,k}})\b}\\
 & = & \TNorm{\matA\x_{k}-\b}\pm d_{2}(\matU_{\mat A\matR,k},\matU_{\mat A,k})\cdot\TNorm{\b}
\end{eqnarray*}
In the first part of the theorem, we have $d_{2}\left(\matU_{\mat A\matR,k},\matU_{\mat A,k}\right)\leq\nu$,
while the second part of the theorem we have $\matU_{\matA\matR,k}=\matU_{\matA\matR}$
(since $\matA\matR$ has $k$ columns) and Lemma~\ref{lem:14} ensures
that $d_{2}\left(\matU_{\mat A\matR,k},\matU_{\mat A,k}\right)\leq\nu\sigma_{k+1}/\sigma_{k}$.
Either way, the additive error bounds of the theorem are met.

We now bound the infeasibility of the approximate solution for the
PCP guarantee (first part of the theorem):
\begin{eqnarray*}
\TNorm{\matU_{\mat A,k+}^{\T}\matA\x_{\matR,k}} & = & \TNorm{\matU_{\mat A,k+}^{\T}\matA\x_{\matR,k}-\matU_{\mat A,k+}^{\T}\matA\x_{k}+\matU_{\mat A,k+}^{\T}\matA\x_{k}}\\
 & \leq & \TNorm{\matU_{\mat A,k+}^{\T}\left(\matA\x_{\matR,k}-\matA\x_{k}\right)}+\TNorm{\matU_{\mat A,k+}^{\T}\matA\x_{k}}\\
 & \leq & \TNorm{\matA\x_{\matR,k}-\matA\x_{k}}\\
 & \leq & d_{2}\left(\matU_{\mat A\matR,k},\matU_{\mat A,k}\right)\TNorm{\b}\\
 & \leq & \nu\TNorm{\b}
\end{eqnarray*}
where we used the fact that $\matA\x_{k}\in\range{\matU_{\matA,k}}$
so $\matU_{\mat A,k+}^{\T}\matA\x_{k}=0$.

We now bound the infeasibility of the approximate solution for the
PCR guarantee (second part of the theorem):

\begin{eqnarray*}
\TNorm{\matV_{\mat A,k+}^{\T}\x_{\matR}} & = & \TNorm{\matV_{\mat A,k+}^{\T}\matR(\matA\matR)^{+}\b}\\
 & \leq & \TNorm{\matV_{\mat A,k+}^{\T}\matR}\cdot\TNorm{(\matA\matR)^{+}}\cdot\TNorm{\b}\\
 & \leq & \frac{\nu}{\left(\sqrt{1-\nu^{2}}-\nu\right)\sigma_{k}}\TNorm{\b}
\end{eqnarray*}
where we used Lemma~\ref{lem:11} to bound $\TNorm{\matV_{\mat A,k+}^{\T}\matR}$
and $\TNorm{(\matA\matR)^{+}}$.
\end{proof}

\subsection{\label{subsec:statistical}Statistical Perspective}

We now consider $\x_{\matR,k}$ from a statistical perspective. We
use a similar framework to the one used in the literature to analyze
CLS~\cite{slawski2017compressed,Slawski17,THM17}. That is, we consider
a fixed design setting in which the rows of $\matA$, $\a_{1},\dots,\a_{n}\in\R^{d}$,
are considered as fixed, and $\b$'s entries, $b_{1},\dots,b_{n}\in\R$,
are 
\[
b_{i}=f_{i}+\xi_{i}
\]
where $f_{1},\dots,f_{n}$ are fixed values and the noise terms $\xi_{1},\dots,\xi_{n}$
are assumed to be independent random values with zero mean and $\sigma^{2}$
variance. We denote by $\f\in\R^{n}$ the vector whose $i$th entry
is $f_{i}$ . The goal is to recover $\f$ from $\b$ (i.e., de-noise
$\b$). 

The optimal predictor $\matA\x^{\star}$ of $\f$ given $\matA$ is
a minimizer of 
\[
\min_{\x\in\R^{d}}\Expect{\TNormS{\matA\x-\b}/n}
\]
where here, and in subsequent expressions, the expectation is with
respect to the noise $\xi$ (if there are multiple minimizers, $\x^{\star}$
is the minimizer with minimum norm). It is easy to verify that $\matA\x^{\star}=\matP_{\matA}\f$.

Given an estimator $\theta=\theta(\matA,\b)$ of $\x^{\star}$ (which
we assume is a random variable since $\b$ is a random varaible),
its \emph{excess risk} is define as 
\[
{\cal E}(\theta)\coloneqq\Expect{\TNormS{\matA\theta-\matA\x^{\star}}/n}\,.
\]
The \emph{ordinary least square estimator }(OLS) $\hat{\x}$ is simply
a solution to $\min_{\x\in\R^{d}}\TNorm{\matA\x-\b}$: $\hat{\x}\coloneqq\matA^{+}\b$.
Simple calculations show that 
\[
{\cal E}(\hat{\x})=\sigma^{2}\rank{\matA}/n\,.
\]
Thus, if the rank of $\matA$ is large, which is usually the case
when $d\gg n$, then the excess risk might be large (and it does not
asymptotically converge to $0$ if $\rank{\matA}=\Omega(n)$). This
motivates the use of regularization (e.g., PCR). Indeed, the excess
risk of the PCR estimator $\x_{k}$ can be bounded~\cite{Slawski17}:
\begin{equation}
{\cal E}(\x_{k})\leq\frac{\InfNormS{\matV_{\matA}^{\T}\x^{\star}}\cdot\sum_{i=k+1}^{\min(n,d)}\sigma_{i}^{2}}{n}+\frac{\sigma^{2}k}{n}\,.\label{eq:pcr-old-risk-bound}
\end{equation}
In many scenarios, $\x_{k}$ has a significantly reduced excess risk
in comparison to the excess risk of $\hat{\x}$ (see~\cite{Slawski17}
for a discussion). This motivates the use of PCR when $d$ is large.

In this section, we analyze the excess risk of $\x_{\matR,k}$ based
on properties of $\matR$. The bounds are based on the following identity~\cite{Slawski17}\footnote{However, no proof of (\ref{eq:bias-var}) appears in~\cite{Slawski17},
so for completeness we include a proof in the appendix. }: for any $\matM$ of appropriate size
\begin{equation}
{\cal E}(\x_{\matM})={\cal E}(\matM(\matA\matM)^{\pinv}\b)=\underset{{\cal B}(\x_{\matM})}{\underbrace{\frac{1}{n}\TNormS{(\matI-\matP_{\matA\matM})\matA\x^{\star}}}}+\underset{{\cal V}(\x_{\matM})}{\underbrace{\sigma^{2}\frac{\rank{\matA\matM}}{n}}}\,.\label{eq:bias-var}
\end{equation}
In the above, ${\cal B}(\x_{\matM})$ can be viewed as a bias term,
and ${\cal V}(\x_{\matM})$ can be viewed as a variance term. Eq.~(\ref{eq:pcr-old-risk-bound})
is obtained by bounding the bias term ${\cal B}(\x_{k})$, although
our results lead to a bound on ${\cal E}(\x_{k})$ that is tighter
in some cases (Corollary~\ref{cor:new-pcr-risk}). An immediate corollary
of~(\ref{eq:bias-var}) is the following bound for $\x_{\matR,k}$:
\begin{equation}
{\cal E}(\x_{\matR,k})=\frac{1}{n}\TNormS{(\matI-\matP_{\matA\matR\matV_{\matA\matR,k}})\matA\x^{\star}}+\frac{\sigma^{2}k}{n}\,.\label{eq:pcp_risk}
\end{equation}

The following results addresses the case where $\matR$ has $k$ orthonormal
columns. The conditions are the same as the first part of Theorem~\ref{thm:structural}
(optimization perspective analysis). 
\begin{thm}
\label{thm:stat-structural}Assume that $\rank{\matA}\geq k$. Suppose
that $\matR\in\R^{d\times k}$ has orthonormal columns, and that $d_{2}(\matR,\matV_{\matA,k})\leq\nu(1+\nu^{2})^{-1/2}<1$.
Then,
\[
{\cal E}(\x_{\matR})\leq\frac{\left(1+\nu\right)\cdot\TNormS{\x^{\star}}\cdot\sigma_{k+1}^{2}}{n}+\frac{\sigma^{2}k}{n}
\]
\end{thm}

For the proof, we need the following theorem due to Halko et al.~\cite{halko2011finding}.
\begin{thm}
[Theorem 9.1 in \cite{halko2011finding}]\label{thm:hmt_structural}
Let $\matA$ be an $m\times n$ matrix with singular value decomposition
$\matA=\matU_{\matA}\mat{\Sigma}_{\matA}\matV_{\matA}^{\matT}$ .Let
$k\ge0$. Let $\matR$ be any matrix such that $\matV_{\matA,k}^{\matT}\matR$
has full row rank. Then we have
\[
\TNormS{(\matI_{m}-\matP_{\matA\matR})\matA}\leq\TNormS{\mat{\Sigma}_{\matA,k+}}+\TNormS{\mat{\Sigma}_{\matA,k+}\matV_{\matA,k+}^{\T}\matR\left(\matV_{\matA,k}^{\T}\matR\right)^{+}}
\]
\end{thm}

\begin{proof}
[Proof of Theorem~\ref{thm:stat-structural}]The condition that $d_{2}(\matR,\matV_{\matA,k})\leq\nu(1+\nu^{2})^{-1/2}<1$
ensures that $\matV_{\matA,k}^{\T}\matR$ has full rank, and that
$\TNormS{\tan\Theta(\matR,\matV_{\matA,k})}\leq\nu$ (since $\Theta(\matR,\matV_{\matA,k})$
is a diagonal matrix whose diagonal values are the inverse cosine
of the singular values of $\matR^{\T}\matV$, and these, in turn,
are all larger than $\sqrt{1-\nu^{2}(1+\nu^{2})^{-1}}$). Thus we
have,

\begin{eqnarray*}
{\cal B}(\x_{\matR}) & = & \frac{1}{n}\TNormS{(\matI-\matP_{\matA\matR})\matA\x^{\star}}\\
 & \leq & \frac{1}{n}\TNormS{\x^{\star}}\cdot\left(\TNormS{\mat{\Sigma}_{\matA,k+}}+\TNormS{\mat{\Sigma}_{\matA,k+}\matV_{\matA,k+}^{\matT}\matR\left(\matV_{\matA,k}^{\matT}\matR\right)^{+}}\right)\\
 & \leq & \frac{1}{n}\TNormS{\x^{\star}}\left(\sigma_{k+1}^{2}+\sigma_{k+1}^{2}\TNormS{\matV_{\matA,k+}^{\matT}\matR\left(\matV_{\matA,k}^{\matT}\matR\right)^{+}}\right)\\
 & = & \frac{1}{n}\TNormS{\x^{\star}}\left(\sigma_{k+1}^{2}+\sigma_{k+1}^{2}\TNormS{\tan\Theta\left(\matR,\matV_{\matA,k}\right)}\right)\\
 & \leq & \frac{\left(1+\nu\right)\cdot\TNormS{\x^{\star}}\cdot\sigma_{k+1}^{2}}{n}
\end{eqnarray*}
where in the first inequality we used Theorem~\ref{thm:hmt_structural}
and for the second equality we used Lemma~\ref{lem:tan-to-spectral}.
The result now follows from the fact that $\rank{\matA\matR}\leq k$.
\end{proof}
\begin{cor}
\label{cor:new-pcr-risk}For the PCR solution $\x_{k}$ we have
\[
{\cal E}(\x_{k})\leq\frac{\TNormS{\x^{\star}}\cdot\sigma_{k+1}^{2}}{n}+\frac{\sigma^{2}k}{n}\,.
\]
\end{cor}

Next, we consider the general case where $\matR$ does not necessarily
have orthonormal columns, and potentially has more than $k$ columns.
The conditions are the same as the second part of Theorem~\ref{thm:structural}
(optimization perspective). 
\begin{thm}
\label{thm:struct-stat-pcp}Suppose that $\matR\in\R^{d\times s}$
where $s\ge k$. Assume that $\rank{\matA\matR}\geq k$. If $d_{2}\left(\matU_{\mat A\matR,k},\matU_{\mat A,k}\right)\leq\nu<1$
then,
\[
{\cal E}(\x_{\matR,k})\leq{\cal E}(\x_{k})+\frac{(2\nu+\nu^{2})\TNormS{\f}}{n}\,.
\]
\end{thm}

\begin{proof}
Since $\matA\matR$ has rank at least $k$, we have $\matP_{\matA\matR\matV_{\matA\matR,k}}=\matP_{\matU_{\matA\matR,k}}$.
From (\ref{eq:pcp_risk}), the fact that $\matA\x^{\star}=\matP_{\matA}\f$,
and $\matP_{\matA\matR\matV_{\matA\matR,k}}\matP_{\matA}=\matP_{\matA\matR\matV_{\matA\matR,k}}$
(since the range of $\matA\matR\matV_{\matA\matR,k}$ is contained
in the range of $\matA$) we have

\begin{eqnarray*}
{\cal B}(\x_{\matR,k}) & = & \frac{1}{n}\TNormS{(\matI-\matP_{\matA\matR\matV_{\matA\matR,k}})\matA\x^{\star}}\\
 & = & \frac{1}{n}\TNormS{(\matP_{\matA}-\matP_{\matA\matR\matV_{\matA\matR,k}})\f}\\
 & = & \frac{1}{n}\TNormS{(\matP_{\matA}-\matP_{\matU_{\matA,k}}+\matP_{\matU_{\matA,k}}-\matP_{\matA\matR\matV_{\matA\matR,k}})\f}\\
 & = & \frac{1}{n}\left(\TNormS{(\matP_{\matA}-\matP_{\matU_{\matA,k}})\f}+\TNormS{(\matP_{\matU_{\matA,k}}-\matP_{\matU_{\matA\matR,k}})\f}+2(\matP_{\matA}\f-\matP_{\matU_{\matA,k}}\f)^{\T}(\matP_{\matU_{\matA,k}}\f-\matP_{\matA\matR\matV_{\matA\matR,k}}\f)\right)\\
 & \leq & {\cal B}(\x_{k})+d_{2}\left(\matU_{\mat A\matR,k},\matU_{\mat A,k}\right)^{2}\frac{\TNormS{\f}}{n}+\frac{2}{n}\left|\f^{\T}(\matP_{\matA}-\matP_{\matU_{\matA,k}})^{\T}(\matP_{\matU_{\matA,k}}-\matP_{\matA\matR\matV_{\matA\matR,k}})\f\right|
\end{eqnarray*}
For the cross-terms, we bound

\begin{eqnarray*}
\left|\f^{\T}(\matP_{\matA}-\matP_{\matU_{\matA,k}})^{\T}(\matP_{\matU_{\matA,k}}-\matP_{\matA\matR\matV_{\matA\matR,k}})\f\right| & = & \left|\f^{\T}\left(\matP_{\matA}\matP_{\matU_{\matA,k}}-\matP_{\matA}\matP_{\matA\matR\matV_{\matA\matR,k}}-\matP_{\matU_{\matA,k}}\matP_{\matU_{\matA,k}}+\matP_{\matU_{\matA,k}}\matP_{\matA\matR\matV_{\matA\matR,k}}\right)\f\right|\\
 & = & \left|\f^{\T}\left(\matP_{\matU_{\matA,k}}-\matP_{\matU_{\matA\matR,k}}-\matP_{\matU_{\matA,k}}+\matP_{\matU_{\matA,k}}\matP_{\matU_{\matA\matR,k}}\right)\f\right|\\
 & = & \f^{\T}\left(\matI-\matP_{\matU_{\matA,k}}\right)\matP_{\matU_{\matA\matR,k}}\f\\
 & = & \f^{\T}\matP_{\matU_{\matA,k}}^{\perp}\matP_{\matU_{\matA\matR,k}}\f\\
 & \leq & \TNorm{\matP_{\matU_{\matA,k}}^{\perp}\matP_{\matU_{\matA\matR,k}}}\cdot\TNormS{\f}
\end{eqnarray*}
Since both $\matU_{\mat A\matR,k}$ and $\matU_{\mat A,k}$ are full
rank,\textcolor{red}{{} }we have (Theorem~\ref{thm:projection_equality})
\[
\TNorm{\matP_{\matU_{\matA,k}}^{\perp}\matP_{\matU_{\matA\matR,k}}}=\TNorm{\matP_{\matU_{\matA,k}}-\matP_{\matU_{\matA\matR,k}}}=d_{2}\left(\matU_{\mat A\matR,k},\matU_{\mat A,k}\right)
\]
Thus, we find that 
\[
{\cal B}(\x_{\matR,k})\leq{\cal B}(\x_{k})+(2\nu+\nu^{2})\frac{\TNormS{\f}}{n}\,.
\]
We reach the bound in the theorem statement by adding the variance
${\cal V}(\x_{\matR,k})$, which is equal to the variance of $\x_{k}$
because the ranks are equal. 
\end{proof}

\paragraph{Discussion. }

Theorem~\ref{thm:stat-structural} shows that if $\matR$ is a good
approximation to $\matV_{\matA,k}$, then there is a small relative
increase to the bias term, while the variance term does not change.
Since we are mainly interested in keeping the asymptotic behavior
of the excess risk (as $n$ goes to infinity), a fixed $\nu$ of modest
value suffices. However, for this result to hold, $\matR$ has to
have exactly $k$ columns and those columns should be orthonormal.
Without these restrictions, we need to resort to Theorem~\ref{thm:struct-stat-pcp}.
In that theorem, we get (if the conditions are met) only an additive
increase in the bias term. Thus if, for example, $\TNormS{\f}/n\to c$
as $n\to\infty$ for some constant $c$, then $\nu$ should tend to
$0$ as $n$ goes to infinity, but a constant value should suffice
if $n$ is fixed. 

\section{\label{sec:alg}Sketched PCR and PCP}

In the previous section, we considered general conditions on $\matR$
which ensure that $\x_{\matR,k}$ is an approximate solution to the
PCR/PCP problem. In this section, we propose algorithms to generate
$\matR$ for which these conditions hold. The main technique we employ
is \emph{matrix sketching}. The idea is to first multiply the data
matrix $\matA$ by some random transformation (e.g., a random projection),
and extract an approximate subspace from the compressed matrix. 

\subsection{Dimensionality Reduction using Sketching }

The compression (multiplication by a random matrix) alluded in the
previous paragraph can be applied either from the left side, or the
right side, or both. In left sketching, which is more appropriate
if the input matrix has many rows and a modest amount of columns,
we propose to use $\matR=\matV_{\matS\matA,k}$ where $\matS$ is
some sketching matrix (we discuss a couple of options shortly). In
right sketching, which is more appropriate if the input matrix has
many columns and a modest amount of rows, we propose to use $\matR=\matG^{\T}$
where $\matG$ is some sketching matrix. Two sided sketching, $\matR=\matG^{\T}\matV_{\matS\matA\matG^{\T},k}$
, is aimed for the case that the number of columns and the number
of rows are large.

The sketching matrices, $\matS$ and $\matG$, are randomized dimensionality
reduction transformations. Quite a few sketching transforms have been
proposed in the literature in recent years. For concreteness, we consider
two specific cases, though our results hold for other sketching transformations
as well (albeit some modifications in the bounds might be necessary).
The first, which we refer to as 'subgaussian map', is a random matrix
in which every entry of the matrix is sampled i.i.d from some subgaussian
distribution (e.g. $N(0,1)$) and the matrix is appropriately scaled
(however, scaling is not necessary in our case). The second transform
is a sparse embedding matrix, in which each column is sampled uniformly
and independently from the set of scaled identity vectors and multiplied
by a random sign. We refer to such a matrix as a \noun{CountSketch}
matrix~\cite{charikar2002finding,woodruff2014sketching}. 

Both transformations described above, and a few other, have, provided
enough rows are used, with high probability the following property
which we refer to as \emph{approximate Gram property}. 
\begin{defn}
Let $\matX\in\R^{m\times n}$ be a fixed matrix. For $\epsilon,\delta\in(0,1/2)$,
a distribution ${\cal D}$ on matrices with $m$ columns has the \emph{$(\epsilon,\delta)$-approximate
Gram matrix} property for $\matX$ if 
\[
\Pr_{\matS\sim{\cal D}}\left(\TNorm{\matX^{\T}\matS^{\T}\matS\matX-\matX^{\T}\matX}\geq\epsilon\TNormS{\matX}\right)\leq\delta\,.
\]
\end{defn}

Recent results by Cohen et al.~\cite{cohen2015optimal}\footnote{Theorem 1 in \cite{cohen2015optimal} with $k=\sr{\matX}.$}
show that when $\matS$ has independent subgaussian entries, then
as long as the number of rows in $\matS$ is $\Omega((\sr{\matX}+\log(1/\delta))/\epsilon^{2})$
then we have $(\epsilon,\delta)$-approximate Gram property for $\matX$.
If $\matS$ is a \noun{CountSketch} matrix, then as long as the number
of rows in $\matS$ is $\Omega(\sr{\matX}^{2}/(\epsilon^{2}\delta))$
then we have $(\epsilon,\delta)$-approximate Gram property for $\matX$~\cite{cohen2015optimal}.

We first describe our results for the various modes of sketching,
and then discuss algorithmic issues and computational complexity. 
\begin{thm}
[Left Sketching] Let $\nu,\delta\in(0,1/2)$ and denote 
\[
\epsilon=\frac{\nu(1+\nu^{2})^{-1/2}}{1+\nu(1+\nu^{2})^{-1/2}}\cdot\gap{\matA}k.
\]
Suppose that $\matS$ is sampled from a distribution that provides
a $(\epsilon,\delta)$-approximate Gram matrix for $\matA$. Then
for $\matR=\matV_{\matS\matA,k}$, with probability $1-\delta$, the
approximate solution $\x_{\matR}$ is a $\left(\frac{\sigma_{k+1}}{\sigma_{k}}\nu,\frac{\nu}{\left(\sqrt{1-\nu^{2}}-\nu\right)\sigma_{k}}\right)$-approximate
PCR and 
\[
{\cal E}(\x_{\matR})\leq\frac{\left(1+\nu\right)\cdot\TNormS{\x^{\star}}\cdot\sigma_{k+1}^{2}}{n}+\frac{\sigma^{2}k}{n}
\]

Thus if, for example, $\matS$ is a \noun{CountSketch} matrix, then
the conditions are met when the number of rows in $\matS$ is 
\[
\Omega\left(\frac{\sr{\matA}^{2}}{\gap{\matA}k^{2}\nu^{2}\delta}\right)
\]
 rows. In another example, if $\matS$ is a subgaussian map, then
the conditions are met when the number of rows in $\matS$ is 
\[
\Omega\left(\frac{\sr{\matA}+\log(1/\delta)}{\gap{\matA}k^{2}\nu^{2}}\right)\,.
\]
\end{thm}

\begin{proof}
Due to Theorems~\ref{thm:structural} and~\ref{thm:stat-structural},
it suffices to show that that $d_{2}\left(\matR,\matV_{\mat A,k}\right)\leq\nu(1+\nu^{2})^{-1/2}$.
Under the conditions of the theorem, with probability of at least
$1-\delta$ we have $\TNorm{\matA^{\T}\matS^{\T}\matS\matA-\matA^{\T}\matA}\leq\epsilon\TNormS{\matA}$.
If that is indeed the case, $\matA^{\T}\matS^{\T}\matS\matA$ has
rank at least $k$ since $\matA^{\T}\matS^{\T}\matS\matA$ and $\matA^{\T}\matA$
are symmetric matrices and we know that $\sigma_{i}^{2}\left(\matA^{\T}\matS^{\T}\matS\matA\right)=\sigma_{i}^{2}\left(\matA^{\T}\matA\right)\pm\TNorm{\matA^{\T}\matS^{\T}\matS\matA-\matA^{\T}\matA}$
(Weyl's Theorem and the fact that $\sigma_{k}^{2}>\epsilon\sigma_{1}^{2}$).
Furthermore, since $\nu>0$ we have $\epsilon<\gap{\matA}k$, and
Theorem~\ref{thm:sin-theta-1} implies that
\begin{eqnarray*}
d_{2}(\matR,\matV_{\matA,k}) & \leq & \frac{\TNorm{\matA^{\T}\matA-\matA^{\T}\matS^{\T}\matS\matA}}{(\sigma_{k}^{2}-\sigma_{k+1}^{2})-\TNorm{\matA^{\T}\matA-\matA^{\T}\matS^{\T}\matS\matA}}\\
 & \leq & \frac{\epsilon}{\gap{\matA}k-\epsilon}\\
 & \leq & \nu(1+\nu^{2})^{-1/2}\,.
\end{eqnarray*}
Thus, we have shown that with probability $1-\delta$ we have $d_{2}\left(\matR,\matV_{\mat A,k}\right)\leq\nu(1+\nu^{2})^{-1/2}$,
as required.
\end{proof}
\begin{thm}
[Right Sketching]\label{thm:right_sketching}Let $\nu,\delta\in(0,1/2)$
and denote 
\[
\epsilon=\frac{\nu}{1+\nu}\cdot\gap{\matA}k.
\]
Suppose that $\matG$ is sampled from a distribution that provides
a $(\epsilon,\delta)$-approximate Gram matrix for $\matA^{\T}$.
Then for $\matR=\matG^{\T}$, with probability $1-\delta$, the approximate
solution $\matA\x_{\matR,k}$ is an $(\nu,\nu)$-approximate PCP and
\[
{\cal E}(\x_{\matR,k})\leq{\cal E}(\x_{k})+\frac{(2\nu+\nu^{2})\TNormS{\f}}{n}
\]

Thus if, for example, $\matG$ is a \noun{CountSketch} matrix, then
the conditions are met when the number of rows in $\matG$ is 
\[
\Omega\left(\frac{\sr{\matA}^{2}}{\gap{\matA}k^{2}\nu^{2}\delta}\right)
\]
 rows. In another example, if $\matG$ is a subgaussian map, then
the conditions are met when the number of rows in $\matG$ is 
\[
\Omega\left(\frac{\sr{\matA}+\log(1/\delta)}{\gap{\matA}k^{2}\nu^{2}}\right)\,.
\]
\end{thm}

\begin{proof}
Due to Theorem~\ref{thm:structural} it suffices to show that that
$d_{2}\left(\matU_{\mat A\matR,k},\matU_{\mat A,k}\right)\leq\nu$.
Under the conditions of the Theorem, with probability of at least
$1-\delta$ we have $\TNorm{\matA\matG^{\T}\matG\matA^{\T}-\matA\matA^{\T}}\leq\epsilon\TNormS{\matA}$.
If that is indeed the case, $\matA\matG^{\T}\matG\matA^{\T}$ has
rank at least $k$ since $\matA\matG^{\T}\matG\matA^{\T}$ and $\matA\matA^{\T}$
are symmetric matrices and we know that $\sigma_{i}^{2}\left(\matA\matG^{\T}\matG\matA^{\T}\right)=\sigma_{i}^{2}\left(\matA\matA^{\T}\right)\pm\TNorm{\matA\matG^{\T}\matG\matA^{\T}-\matA\matA^{\T}}$
(Weyl's Theorem and the fact that $\sigma_{k}^{2}>\epsilon\sigma_{1}^{2}$).
Furthermore, since $\nu>0$ we have $\epsilon<\gap{\matA}k$, and
Theorem~\ref{thm:sin-theta-1} implies
\begin{eqnarray*}
d_{2}(\matU_{\mat A\matR,k},\matU_{\mat A,k}) & \leq & \frac{\TNorm{\matA\matG^{\T}\matG\matA^{\T}-\matA\matA^{\T}}}{(\sigma_{k}^{2}-\sigma_{k+1}^{2})-\TNorm{\matA\matG^{\T}\matG\matA^{\T}-\matA\matA^{\T}}}\\
 & \leq & \frac{\epsilon}{\gap{\matA}k-\epsilon}\\
 & \leq & \nu\,.
\end{eqnarray*}
Thus, we have shown that with probability $1-\delta$ we have $d_{2}\left(\matU_{\mat A\matR,k},\matU_{\mat A,k}\right)\leq\nu$,
as required.
\end{proof}
\begin{thm}
[Two Sided Sketching]\label{thm:left_right_sketching}Let $\nu,\delta\in(0,1/2)$
and denote 
\[
\epsilon_{2}=\frac{\nu}{2(1+\nu/2)}\cdot\gap{\matA}k.
\]
Suppose that $\matG$ is sampled from a distribution that provides
a $(\epsilon_{2},\delta/2)$-approximate Gram matrix for $\matA^{\T}$.
Denote
\[
\epsilon_{1}=\frac{\nu(1+\nu^{2}/4)^{-1/2}/2}{1+\nu(1+\nu^{2}/4)^{-1/2}/2}\cdot\gap{\matA\matG^{\T}}k
\]
and uppose that $\matS$ is sampled from a distribution that provides
a $(\epsilon_{1},\delta/2)$-approximate Gram matrix for $\matA\matG^{\T}$.
Then for $\matR=\matG^{\T}\matV_{\matS\matA\matG^{\T},k}$ with probability
$1-\delta$ the approximate solution $\matA\x_{\matR,k}$ is an $(\nu,\nu)$-approximate
PCP and 
\[
{\cal E}(\x_{\matR,k})\leq{\cal E}(\x_{k})+\frac{(2\nu+\nu^{2})\TNormS{\f}}{n}
\]

Thus if, for example, $\matS$ is a \noun{CountSketch} matrix and
$\matG$ is a subgaussian map, then the conditions hold when the numbers
of rows of $\matS$ is 
\[
\Omega\left(\frac{\sr{\matA\matG^{\T}}^{2}}{\gap{\matA\matG^{\T}}k^{2}\nu^{2}\delta}\right)
\]
 and the number of rows in $\matG$ is 
\[
\Omega\left(\frac{\sr{\matA}+\log(1/\delta)}{\gap{\matA}k^{2}\nu^{2}}\right)\,.
\]
\end{thm}

\begin{proof}
Due to Theorem~\ref{thm:structural} it suffices to show that that
$d_{2}\left(\matU_{\mat A\matR,k},\matU_{\mat A,k}\right)\leq\nu$. 

Under the conditions of the Theorem, with probability of at least
$1-\delta/2$ we have $\TNorm{(\matA\matG^{\T})^{\T}\matS^{\T}\matS\matA\matG^{\T}-(\matA\matG^{\T})^{\T}\matA\matG^{\T}}\leq\epsilon_{1}\TNormS{\matA\matG^{\T}}$,
and with probability of at least $1-\delta/2$ we have $\TNorm{\matA\matG^{\T}\matG\matA^{\T}-\matA\matA^{\T}}\leq\epsilon_{2}\TNormS{\matA}$.
Thus, both inequalities hold with probability of at least $1-\delta$.
If that is indeed the case, $\matA\matG^{\T}\matG\matA^{\T}$ has
rank at least $k$ since $\matA\matG^{\T}\matG\matA^{\T}$ and $\matA\matA^{\T}$
are symmetric matrices and we know that $\sigma_{i}^{2}\left(\matA\matG^{\T}\matG\matA^{\T}\right)=\sigma_{i}^{2}\left(\matA\matA^{\T}\right)\pm\TNorm{\matA\matG^{\T}\matG\matA^{\T}-\matA\matA^{\T}}$
(Weyl's Theorem and the fact that $\sigma_{k}^{2}>\epsilon_{2}\sigma_{1}^{2}$).
Moreover, $(\matA\matG^{\T})^{\T}\matS^{\T}\matS\matA\matG^{\T}$
has rank at least $k$ since $(\matA\matG^{\T})^{\T}\matS^{\T}\matS\matA\matG^{\T}$
and $\matA\matG^{\T}\matG\matA^{\T}$ are symmetric matrices and we
know that $\sigma_{i}^{2}\left((\matA\matG^{\T})^{\T}\matS^{\T}\matS\matA\matG^{\T}\right)=\sigma_{i}^{2}\left(\matA\matG^{\T}\matG\matA^{\T}\right)\pm\TNorm{(\matA\matG^{\T})^{\T}\matS^{\T}\matS\matA\matG^{\T}-\matA\matG^{\T}\matG\matA^{\T}}$
(Weyl's Theorem and the fact that $\sigma_{k}^{2}\left(\matA\matG^{\T}\right)>\epsilon_{1}\sigma_{1}^{2}\left(\matA\matG^{\T}\right)$).
Since $\nu>0$ we have $\epsilon_{1}<\gap{\matA\matG^{\T}}k$ and
Theorem~\ref{thm:sin-theta-1} implies

\begin{eqnarray*}
d_{2}(\matV_{\matA\matG^{\T}\matV_{\matS\matA\matG^{\T},k}},\matV_{\matA\matG^{\T},k}) & \leq & \frac{\TNorm{(\matA\matG^{\T})^{\T}\matS^{\T}\matS\matA\matG^{\T}-(\matA\matG^{\T})^{\T}\matA\matG^{\T}}}{\sigma_{k}^{2}\left(\matA\matG^{\T}\right)-\sigma_{k+1}^{2}\left(\matA\matG^{\T}\right)-\TNorm{(\matA\matG^{\T})^{\T}\matS^{\T}\matS\matA\matG^{\T}-(\matA\matG^{\T})^{\T}\matA\matG^{\T}}}\\
 & \leq & \frac{\epsilon_{1}\TNormS{\matA\matG^{\T}}}{\sigma_{k}^{2}\left(\matA\matG^{\T}\right)-\sigma_{k+1}^{2}\left(\matA\matG^{\T}\right)-\epsilon_{1}\TNormS{\matA\matG^{\T}}}\\
 & = & \frac{\epsilon_{1}}{\gap{\matA\matG^{\T}}k-\epsilon_{1}}\\
 & \leq & \nu(1+\nu^{2}/4)^{-1/2}/2
\end{eqnarray*}
From Lemma~\ref{lem:14} (with $\matA\matG^{\T}$) we get that $d_{2}(\matU_{\matA\matG^{\T}\matV_{\matS\matA\matG^{\T},k}},\matU_{\matA\matG^{\T},k})\leq\frac{\sigma_{k+1}(\matA\matG^{\T})}{\sigma_{k}(\matA\matG^{\T})}(\nu/2)\leq\nu/2$. 

We now bound
\begin{eqnarray*}
d_{2}\left(\matU_{\mat A\matR,k},\matU_{\mat A,k}\right) & \leq & d_{2}\left(\matU_{\mat A\matR,k},\matU_{\mat A\matG^{\T},k,k}\right)+d_{2}\left(\matU_{\mat A\matG^{\T},k},\matU_{\mat A,k}\right)\\
 & \leq & \nu
\end{eqnarray*}
where we similarly use Theorem \ref{thm:right_sketching} to bound
$d_{2}(\matU_{\mat A\matG^{\T},k},\matU_{\mat A,k})\leq\nu/2.$

Thus, we have shown that with probability $1-\delta$ we have $d_{2}\left(\matU_{\mat A\matR,k},\matU_{\mat A,k}\right)\leq\nu$,
as required.
\end{proof}

\subsection{Fast Approximate PCR/PCP}

A prototypical algorithm for approximate PCR/PCP is to compute $\x_{\matR,k}$
with some choice of sketching-based $\matR$. There are quite a few
design choices that need to be made in order to turn this prototypical
algorithm into a concrete algorithm, e.g. whether to use left, right
or two sided sketching to form $\matR$, and which sketch transform
to use. There are various tradeoffs, e.g. using \noun{CountSketch}
results in faster sketching, but usually requires larger sketch sizes.
Furthermore, in computing $\x_{\matR,k}$ there are also algorithmic
choices to be made with respect to choosing the order of matrix multiplications:
in computing $(\matA\matR\matV_{\matA\matR,k})^{\pinv}\b$ should
we first compute $\matA\matR$ and then multiply by $\matV_{\matA\matR,k}$,
or vice versa? Likely, there is no one size fit all algorithm, and
different profiles of the input matrix (in particular, the size and
sparsity level) call for a different variant of the prototypical algorithm. 

\begin{table}
\begin{centering}
{\footnotesize{}}%
\begin{tabular}{|c|c|c|c|}
\hline 
 &  & {\footnotesize{}$\x_{\matR,k}$} & {\footnotesize{}CLS}\tabularnewline
\hline 
\hline 
{\footnotesize{}Left sketching} & {\footnotesize{}subgaussian $\matS$} & {\footnotesize{}$s_{1}\cdot\nnz{\matA}+s_{1}d\min(s_{1},d)+nk^{2}$} & {\footnotesize{}N/A}\tabularnewline
\cline{2-4} 
{\footnotesize{}$\matR=\matV_{\matS\matA,k}$} & \noun{\footnotesize{}CountSketch}{\footnotesize{} $\matS$} & {\footnotesize{}$k\cdot\nnz{\matA}+s_{2}d\min(s_{2},d)+nk^{2}$} & {\footnotesize{}N/A}\tabularnewline
\hline 
{\footnotesize{}Right sketching} & {\footnotesize{}subgaussian $\matG$} & {\footnotesize{}$t_{1}\cdot\nnz{\matA}+nt_{1}\min(n,t_{1})+t_{1}k\min(n,d)$} & {\footnotesize{}$t_{1}\cdot\nnz{\matA}+nt_{1}\min(n,t_{1})$}\tabularnewline
\cline{2-4} 
{\footnotesize{}$\matR=\matG^{\T}$} & \noun{\footnotesize{}CountSketch}{\footnotesize{} $\matG$} & {\footnotesize{}$\nnz{\matA}+nt_{2}\min(n,t_{2})+t_{2}k\min(n,d)$} & {\footnotesize{}$\nnz{\matA}+nt_{2}\min(n,t_{2})$}\tabularnewline
\hline 
{\footnotesize{}Two sided } & \noun{\footnotesize{}CountSketch} & {\footnotesize{}$\nnz{\matA}+s_{2}k^{2}+k\min(nt_{2},\nnz{\matA})+nk^{2}$} & {\footnotesize{}N/A}\tabularnewline
{\footnotesize{}$\matR=\matG^{\T}\matV_{\matS\matA\matG^{\T},k}$} & {\footnotesize{}$\matG$ and $\matS$} &  & \tabularnewline
\hline 
\end{tabular}
\par\end{centering}{\footnotesize \par}
\caption{\label{tab:complex}Computational complexity of computing $\protect\x_{\protect\matR,k}$
and CLS for various options $\protect\matR$. For brevity, we omit
the $O$$()$ from the notation. }
\end{table}

Table~\ref{tab:complex} summarizes the running time complexity of
several design options. In order to better make sense between these
different choices, we first summarize the running time complexity
of various design choices using the optimal implementation (from an
asymptotic running-time complexity perspective). To make the discussion
manageable, we consider only subgaussian maps and \noun{CountSketch}.
Furthermore, for the sake of the analysis, we make some assumptions
and adopt some notational conventions. First, we assume that computing
$\matB^{\pinv}\c$ for some $\matB\in\R^{m\times n}$ and $\c$ is
done via straightforward methods based on QR or SVD factorizations,
and as such takes $O(mn\min(m,n))$. We consider using fast sketch-based
approximate least squares algorithms in the next subsection. Next,
we let the sketch sizes be parameters in the complexity. In the discussion,
we use our theoretical results to deduce reasonable assumptions on
how these parameters are set, and thus to reason about the final complexity
of sketched PCR/PCP. We denote the number of rows in the left sketch
matrix $\matS$ by $s_{1}$ for a subgaussian map, and $s_{2}$ for
\noun{CountSketch}. We denote the number of rows in the left sketch
matrix $\matG$ by $t_{1}$ for a subgaussian map, and $t_{2}$ for
\noun{CountSketch}. Finally, we assume $\nnz{\matA}\geq\max(n,d)$,
and that all sketch sizes are greater than $k$. 

Table~\ref{tab:complex} also lists, where relevant, the complexity
of the CLS solution $\x_{\matR}$. 

\paragraph{Discussion.}

We first compare the computational complexity of CLS to the computational
complexity of our proposed right sketching algorithm. For both choices
of $\matG$ we have for sketched PCP an additional term of $O(tk\min(n,d))$.
However, close inspection reveals that this term is dominated by the
term $O(nt\min(n,t))$. Thus our proposed algorithm has the same asymptotic
complexity as CLS for the same sketch size. However, our algorithm
does not mix regularization and compression and comes with stronger
theoretical guarantees. 

Next, in order to compare subgaussian maps to \noun{CountSketch},
we first make some simplified assumptions on the required approximation
quality $\nu$, the relative eigengap $\gap{\matA}k$, and the rank
parameter $k$: $\nu$ is fixed, $\gap{\matA}k$ is bounded from below
by a constant, and we have $k=O(\sr{\matA})$. The first assumptions
is justified if we are satisfied with fixed sub-optimality in the
objective (optimization perspective), or a small constant multiplicative
increase in excess risk if left sketching is used, or $n$ is fixed
(statistical perspective). The first assumption is somewhat less justified
from a statistical point of view when $n\to\infty$ and right sketching
is used. The rationale behind the second assumption is that the PCR/PCP
problem is in a sense ill-posed if there is a tiny eigengap. The third
assumption is motivated by the fact that the stable rank is a measure
of the number of large singular values, which are typically singular
values that correspond to the signal rather than noise. With these
assumptions, our theoretical results establish that $s_{1},t_{1}=O(k)$
and $s_{2},t_{2}=O(k^{2})$ suffice. It is important to stress that
we make these assumptions only for the sake of comparing the different
sketching options, and we do not claim that these assumptions always
hold, or that our proposed algorithms work only when these assumptions
hold. 

For left sketching, with these assumptions, we have complexity of
$O(k\nnz{\matA}+k^{2}\max(n,d))$ for subgaussian maps and $O(k\nnz{\matA}+dk^{2}\min(k^{2},d)+nk^{2})$
for \noun{CountSketch}. Clearly, better asymptotic complexity is achieved
with subgaussian maps. For right sketching, with these assumptions,
we have complexity of $O(k\nnz{\matA}+nk^{2})$ for subgaussian sketch
and $O(\nnz{\matA}+nk^{2}\min(n,k^{2}))$ for \noun{CountSketch}.
The complexity in terms of the input sparsity $\nnz{\matA}$\noun{,}
which is arguably the dominant term, is better for \noun{CountSketch}.
For two sided sketching, we have complexity $O(\nnz{\matA}+k^{4}+k\min(nk^{2},\nnz{\matA})+nk^{2}$).

If $n\gg d$ and $\nnz{\matA}=O(n)$ (sparse input matrix, and constant
amount of non zero features per data point), left sketching gives
better asymptotic complexity. If $n\gg d$ and $\nnz{\matA}=nd$ (full
data matrix), left sketching has better complexity unless $d\gg k^{3}$.
Furthermore, left sketching gives stronger theoretical guarantees.\emph{
Thus, for $n\gg d$} \emph{we advocate the use of left sketching}.
If $d\gg n$ and $\nnz{\matA}=O(d)$ (sparse input matrix), right
sketching with a subgaussian maps always has better complexity than
left sketching, and potentially (but not always) right sketching with
\noun{CountSketch} has even better complexity. If $d\gg n$ and $\nnz{\matA}=nd$
(full data matrix), right sketching with subgaussian maps has the
same complexity as left sketching, and potentially (but not always)
right sketching with \noun{CountSketch} has even better complexity
(if $d$ is sufficiently larger than $n$). \emph{Thus, for $d\gg n$
we advocate the use of right sketching}. If $n\approx d$ (both very
large), and $\nnz{\matA}=n$ then it is possible to have $O(nk^{2})$
with all three options (left, right and two sided), as long as $k^{2}\leq n$.
A similar conclusion is achieved if $n\approx d$ and $\nnz{\matA}=nd$,
but if $k^{2}\ll n$ then two sided sketch is better. 

\subsection{\label{subsec:input-sparsity}Input Sparsity Approximate PCP}

In this section, we propose an input-sparsity algorithm for approximate
PCP. In 'input-sparsity algorithm', we mean an algorithm whose running
time is $O(\nnz{\matA}\log(d/\epsilon)+\poly{k,s,t,\log(1/\epsilon})$,
where $\epsilon$ is some accuracy parameter (see formal theorem statement).

\begin{algorithm}
\begin{algorithmic}[1]

\STATE \textbf{Input: $\matA\in\R^{n\times d}$},\textbf{ $\b\in\R^{n}$},\textbf{
}$k\leq\min(n,d)$, $s,t\geq k$, $\epsilon\in(0,1)$ 

\STATE 

\STATE Generate two \noun{CountSketch} matrices $\matS\in\R^{s\times n}$
and $\matG\in\R^{t\times n}$. 

\STATE Remove from $\matG$ any row that is zero. 

\STATE $\matC\gets\matA\matG^{\T}$.

\STATE $\matD\gets\matS\matC$.

\STATE Compute $\matV_{\matD,k}$, the $k$ dominant right invariant
space of $\matG$ (via SVD). 

\STATE For the analysis (no need to compute): $\matR=\matG^{\T}\matV_{\matD,k}$.

\STATE Solve $\min_{\gamma}\TNorm{\matC\matV_{\matD,k}\gamma-\b}$
to $\epsilon/d$ accuracy using input-sparsity least squares regression
(see \cite[section 7.7]{CW17}). (Do not compute $\matC\matV_{\matD,k}$.
In each iteration, multiplying a vector by $\matC\matV_{\matD,k}$
is performed by first multiplying by $\matV_{\matD,k}$ and then by
$\matC$.)

\STATE Return $\y\gets\matG^{\T}(\matV_{\matD,k}\tilde{\gamma})$,
where $\tilde{\gamma}$ is the output of the previous step. 

\end{algorithmic}

\caption{\label{alg:input_sparsity}Input Sparsity Approximate PCP}
\end{algorithm}

The basic idea is to use two sided sketching, with an additional modification
of using input sparsity algorithms to approximate $(\matA\matR)^{\pinv}\b=\arg\min_{\gamma}\TNorm{\matA\matR\gamma-\b}$.
Specifically, we propose to use the algorithm recently suggested by
Clarkson and Woodruff~\cite{CW17}. A pseudo-code description of
our input sparsity approximate PCP algorithm is listed in Algorithm~\ref{alg:input_sparsity}.
We have the following statement about the algorithm.
\begin{thm}
Run Algorithm \ref{alg:input_sparsity} with $\epsilon,s,t,k$ as
parameters. Under exact arithmetic\footnote{The results are likely too optimistic for inexact arithmetic. We leave
the numerical analysis to future work. }, after 
\[
O(\nnz{\matA}\log(d/\epsilon)+\log(d/\epsilon)tk+sk^{2}+t^{3}k+k^{3}\log^{2}k)
\]
operations, with probability $2/3$, the algorithm will return a $\y$
such that 
\[
\TNormS{\y-\x_{\matR}}\leq\epsilon\TNormS{\x_{\matR}}\,.
\]
\end{thm}

\begin{proof}
Denote $\matB=\matA\matR$, and consider using the iterative method
described in~\cite[section 7.7]{CW17} to approximately solve $\min_{\gamma}\TNorm{\matB\gamma-\b}$.
Denote the optimal solution by $\gamma_{\matR}$, and the solution
that our algorithm found by $\tilde{\gamma}$. Theorem 7.14 in \cite{CW17}
states that after the $O(\log(d/\epsilon))$ iterations the algorithm
would have returned $\tilde{\gamma}$ such that 
\begin{equation}
\TNormS{\matB\matZ(\tilde{\gamma}-\gamma_{\matR})}\leq(\epsilon/d)\TNormS{\matB\matZ\gamma_{\matR}}\label{eq:gamma_err}
\end{equation}
for some invertible $\matZ$ found by the algorithm. Furthermore,
$\kappa(\matB\matZ)=O(1)$ where $\kappa(\cdot)$ is the condition
number (ratio between the largest singular value and smallest). Eq.~(\ref{eq:gamma_err})
implies that $\TNormS{\tilde{\gamma}-\gamma_{\matR}}\leq\kappa(\matB\matZ)^{2}(\epsilon/d)\TNormS{\gamma_{\matR}}=O(\epsilon/d)\TNormS{\gamma_{\matR}}$.
Now, noticing that $\x_{\matR}=\matR\gamma_{\matR}$ and $\y=\matR\tilde{\gamma}$,
we find that 
\[
\TNormS{\y-\x_{\matR}}\leq O(\epsilon/d)\kappa(\matR)^{2}\TNormS{\x_{\matR}}\,.
\]
We now need to bound $\kappa(\matR)=\kappa(\matG^{\T}\matV_{\matS\matA\matG^{\T},k})=\kappa(\matG^{\T})$
where $\matG$ is a \noun{CountSketch} matrix. Since $\matG$ has
a single non zero in each column, then $\TNormS{\matG^{\T}}\leq\FNormS{\matG^{\T}}\leq d$.
Furthermore, since we removed zero column from $\matG^{\T}$, for
any $\x$ the vector $\matG^{\T}\x$ has in one of its coordinates
any coordinate of $\x$, so $\sigma_{\min}(\matG^{\T})\geq1$. So
we found that $\kappa(\matG^{\T})^{2}\leq d$. We conclude that 
\[
\TNormS{\y-\x_{\matR}}\leq O(\epsilon)\TNormS{\x_{\matR}}\,.
\]
Adjusting $\epsilon$ to compensate for the constants completes the
proof. 
\end{proof}

\section{\label{sec:extensions}Extensions}

\subsection{Streaming Algorithm}

We now consider computing an approximate PCR/PCP in the streaming
model. We consider a one-pass row-insertion streaming model, in which
the rows of $\matA$, $\a_{1},\a_{2},\dots,\a_{n}$, and the corresponding
entries in $\b$, $b_{1},b_{2},\dots,b_{n}$, are presented one-by-one
and once only (i.e., in a stream). The goal is to use $o(n)$ memory
(so $\matA$ cannot be stored in memory). The relevant resources to
be bounded for numerical linear algebra in the streaming model are
storage, update time (time spent per row), and final computation time
(at the end of the stream)~\cite{CW09}. Our goal is to bound these
by $O(\poly d)$ .

Our proposed streaming algorithm for approximate PCP uses left sketching.
It is easy to verify that if $\matS$ is a subgaussian map or \noun{CountSketch,
}then $\matR=\matV_{\matS\matA,k}$ can be computed in the streaming
model: one has to update $\matS\matA$ as new rows are presented ($O(d)$
update for \noun{CountSketch}, and $O(sd)$ for subgaussian map),
and once the final row has been presented, factorizing $\matS\matA$
and extracting $\matR$ can be done in $O(sd\min(s,d))$ which is
polynomial in $d$ if $s$ is polynomial in $d$. However, to compute
$\x_{\matR}$ one has to compute $(\matA\matR)^{\pinv}\b$, and storing
$\matA\matR$ in memory requires $\Omega(n)$ memory. To circumvent
this issue we propose to introduce another sketching matrix $\matT$,
and approximate $(\matA\matR)^{\pinv}\b$ via $(\matT\matA\matR)^{\pinv}\b$
. Thus, for $\matR=\matV_{\matS\matA,k}$ we approximate $\x_{\matR}$
by $\tilde{\x}_{\matR}=\matR(\matT\matA\matR)^{\pinv}\b$. It is easy
to verify that $\tilde{\x}_{\matR}$ can be computed in the streaming
model (by forming and updating $\matT\matA$ while computing $\matR$).

More generally, for \emph{any} $\matR$ which can be computed in the
streaming model, we can also compute in the streaming model the following
approximation of $\x_{\matR,k}$:

\[
\tilde{\x}_{\matR,k}:=\matR\matV_{\matA\matR,k}(\matT\matA\matR\matV_{\matA\matR,k})^{+}\matT\b\,.
\]
The next theorem, establishes conditions on $\matT$ that guarantee
that $\tilde{\x}_{\matR,k}$ is a an approximate PCR/PCP. 
\begin{thm}
\label{thm:sketch-stream-structure-impl}Suppose that $\matR\in\R^{d\times s}$
with $s\ge k$. Assume that $\nu\in(0,1)$. Suppose that $\matT$
provides a $O(\nu)$-distortion subspace embedding for $\range{\left[\begin{array}{ccc}
\matU_{\matA\matR,k} & \matU_{\matA,k} & \b\end{array}\right]}$ that is
\[
\TNormS{\matU_{\matA\matR,k}\x_{1}+\matU_{\matA,k}\x_{2}+\b x_{3}}=(1\pm O(\nu))(\TNormS{\x_{1}}+\TNormS{\x_{2}}+x_{3}^{3})
\]
for every $\x_{1},\x_{2}\in\R^{k}$ and $x_{3}\in\R$. Then,
\begin{enumerate}
\item If $d_{2}\left(\matU_{\mat A\matR,k},\matU_{\mat A,k}\right)\leq\nu$
then $\matA\tilde{\x}_{\matR,k}$ is an $(O(\nu),O(\nu))$-approximate
PCP. 
\item If $s=k$ and $\matR$ has orthonormal columns (i.e., $\matR^{\T}\matR=\matI_{k})$
and $d_{2}\left(\matR,\matV_{\mat A,k}\right)\leq\nu(1+\nu^{2})^{-1/2}$
then $\tilde{\x}_{\matR}$ is an $(O(\nu),\,O(\nu/\sigma_{k}))$-approximate
PCR. 
\end{enumerate}
The subspace embedding conditions on $\matT$ are met with probability
of at least $1-\delta$ if, for example, $\matT$ is a \noun{CountSketch}
matrix with $O(k^{2}/\nu^{2}\delta)$ rows. 
\end{thm}

\begin{proof}
We need to show both the additive error bounds on the objective function,
and the error bound on the constraints. We start with the additive
error bounds on the objective function for both for PCP (first part
of the theorem) and for PCR (second part of the theorem). The lower
bound on $\TNorm{\matA\tilde{\x}_{\matR,k}-\b}$ follows immediately
from the fact that $\tilde{\x}_{\matR,k}\in\range{\matR\matV_{\matA\matR,k}}$
and the fact that $\x_{\matR,k}$ is a minimizer of $\TNorm{\matA\x-\b}$
subject to $\x\in\range{\matR\matV_{\matA\matR,k}}$. For the upper
bound, we observe 
\begin{eqnarray*}
\TNorm{\matA\tilde{\x}_{\matR,k}-\b} & \leq & \left(1+O(\nu)\right)\TNorm{\matT\matA\matR\matV_{\matA\matR,k}(\matT\matA\matR\matV_{\matA\matR,k})^{+}\matT\b-\matT\b}\\
 & \leq & \left(1+O(\nu)\right)\TNorm{\matT\matA\matR\matV_{\matA\matR,k}(\matA\matR\matV_{\matA\matR,k})^{+}\b-\matT\b}\\
 & \leq & \left(1+O(\nu)\right)\TNorm{\matA\x_{\matR,k}-\b}\\
 & \leq & \TNorm{\matA\x_{\matR,k}-\b}+O(\nu)\TNorm{\b}
\end{eqnarray*}
where in the first and third inequality we used the fact that $\matT$
provides a subspace embedding for $\range{[\matA\matR\matV_{\matA\matR,k}\,\b]}$
and in the second inequality we used the fact that $(\matT\matA\matR\matV_{\matA\matR,k})^{+}\matT\b$
is a minimizer of $\TNorm{\matT\matA\matR\matV_{\matA\matR,k}\x-\matT\b}$.
Bounds on $\TNorm{\matA\x_{\matR,k}-\b}$ (Theorem~\ref{thm:structural})
now imply the additive bound. 

We now bound the constraint for the PCR guarantee (second part of
the theorem). Let
\[
\matC=(\matT\matU_{\matA\matR,k})^{+}((\matT\matU_{\matA\matR,k})^{\T})^{+}\,.
\]
 Since $(\matT\matU_{\matA\matR,k})^{\T}$ and $(\matT\matU_{\matA\matR,k})^{+}$
have the same row space, and $\matT\matU_{\matA\matR,k}$ has more
rows than columns, $\matC$ is non-singular and we have $\matC(\matT\matU_{\matA\matR,k})^{\T}=(\matT\matU_{\matA\matR,k})^{+}$.
Since $\matT$ provides a subspace embedding for $\matU_{\matA\matR,k}$,
all the singular values of $\matT\matU_{\matA\matR,k}$ belong to
the interval $[1-O(\nu),1+O(\nu)]$. We conclude that $\TNorm{\matC-\matI_{k}}\leq O(\nu)$.
We also have $(\matT\matU_{\matA\matR,k}\mat{\Sigma}_{\matA\matR,k})^{+}=\mat{\Sigma}_{\matA\matR,k}^{-1}(\matT\matU_{\matA\matR,k})^{+}$
since $\matT\matU_{\matA\matR,k}$ has linearly independent columns
(since it provides a subspace embedding), and $\mat{\Sigma}_{\matA\matR,k}$
has all linearly independent rows. Thus, 

\begin{eqnarray*}
\TNorm{\matU_{\mat A,k+}^{\T}\matA\tilde{\x}_{\matR,k}} & = & \TNorm{\matU_{\mat A,k+}^{\T}\matA\tilde{\x}_{\matR,k}-\matU_{\mat A,k+}^{\T}\matU_{\mat A,k}\matU_{\mat A,k}^{\T}\matT^{\T}\matT\b}\\
 & \leq & \TNorm{\matA\tilde{\x}_{\matR,k}-\matU_{\mat A,k}\matU_{\mat A,k}^{\T}\matT^{\T}\matT\b}\\
 & = & \TNorm{\matU_{\matA\matR,k}(\matT\matU_{\matA\matR,k})^{+}\matT\b-\matU_{\mat A,k}\matU_{\mat A,k}^{\T}\matT^{\T}\matT\b}\\
 & = & \TNorm{\matU_{\matA\matR,k}\matC(\matT\matU_{\matA\matR,k})^{\T}\matT\b-\matU_{\mat A,k}\matU_{\mat A,k}^{\T}\matT^{\T}\matT\b}\\
 & \leq & \left(1+O(\nu)\right)\cdot\TNorm{\matU_{\matA\matR,k}\matC\matU_{\matA\matR,k}^{\T}\matT^{\T}-\matU_{\mat A,k}\matU_{\mat A,k}^{\T}\matT^{\T}}\cdot\TNorm{\b}\\
 & \leq & \left(1+O(\nu)\right)^{2}\cdot\TNorm{\matU_{\matA\matR,k}\matC\matU_{\matA\matR,k}^{\T}-\matU_{\mat A,k}\matU_{\mat A,k}^{\T}}\TNorm{\b}\\
 & \leq & \left(1+O(\nu)\right)\cdot\left(\TNorm{\matU_{\matA\matR,k}\left(\matC-\matI_{k}\right)\matU_{\matA\matR,k}^{\T}}+\TNorm{\matU_{\matA\matR,k}\matU_{\matA\matR,k}^{\T}-\matU_{\mat A,k}\matU_{\mat A,k}^{\T}}\right)\cdot\TNorm{\b}\\
 & \leq & \left(1+O(\nu)\right)\cdot\left(\TNorm{\matU_{\matA\matR,k}\left(\matC-\matI_{k}\right)\matU_{\matA\matR,k}^{\T}}+\nu\right)\cdot\TNorm{\b}\\
 & = & \left(1+O(\nu)\right)\cdot\left(\TNorm{\left(\matC-\matI_{k}\right)}+\nu\right)\cdot\TNorm{\b}\\
 & \leq & \left(1+O(\nu)\right)\cdot\left(O(\nu)+\nu\right)\cdot\TNorm{\b}\\
 & = & O(\nu)\cdot\TNorm{\b}
\end{eqnarray*}

We now bound the constraint for the PCR guarantee (second part of
the theorem). To that end, and observe:

\begin{eqnarray*}
\TNorm{\matV_{\mat A,k+}^{\T}\tilde{\x}_{\matR,k}} & \leq & \TNorm{\matV_{\mat A,k+}^{\T}\matR\matV_{\matA\matR,k}(\matT\matA\matR\matV_{\matA\matR,k})^{+}\matT\b}\\
 & \leq & \TNorm{\matV_{\mat A,k+}^{\T}\matR}\cdot\TNorm{(\matT\matU_{\matA\matR,k}\mat{\Sigma}_{\matA\matR,k})^{+}\matT\b}\\
 & \leq & \frac{\nu\cdot\left(1+O(\nu)\right)\cdot\TNorm{\b}}{\sigma_{\min}\left(\matT\matU_{\matA\matR,k}\mat{\Sigma}_{\matA\matR,k}\right)}\\
 & \leq & \frac{\nu\cdot\left(1+O(\nu)\right)\TNorm{\b}}{\left(1-O(\nu\right))\sigma_{\min}\left(\matU_{\matA\matR,k}\mat{\Sigma}_{\matA\matR,k}\right)}\\
 & \leq & \frac{O(\nu)\cdot\TNorm{\b}}{\sigma_{\min}\left(\matA\matR\right)}\\
 & \leq & \frac{O(\nu)}{\sigma_{k}}\cdot\TNorm{\b}
\end{eqnarray*}
where we used the fact that $\matT$ provides a subspace embedding
for $\range{\left[\begin{array}{cc}
\matU_{\matA\matR,k} & \b\end{array}\right]}$, and used Lemma~\ref{lem:11} to bound $\TNorm{\matV_{\mat A,k+}^{\T}\matR}$
and $\TNorm{(\matA\matR)^{+}}$.
\end{proof}

\subsection{Approximate Kernel PCR}

For simplicity, we consider only the homogeneous polynomial kernel
${\cal K}(\x,\z)=(\x^{\T}\z)^{q}$. The results trivially extend to
the non-homogeneous polynomial kernel ${\cal K}_{n}(\x,\z)=(\x^{\T}\z+c)^{q}$
by adding a single feature to each data point. We leave to future
work the development of similar techniques for other kernels (e.g.
Gaussian kernel).

Let $\phi:\R^{d}\to\R^{d^{q}}$ be the function that maps a vector
$\z=(z_{1},\dots,z_{d})$ to the set of monomials formed by multiplying
$q$ entries of $\z$, i.e. $\phi(\z)=(z_{i_{1}}z_{i_{2}}\cdots z_{i_{q}})_{i_{1},\dots,i_{q}\in\{1,\dots,d\}}$.
For a data matrix $\matA\in\R^{d}$ and a response vector $\b\in\R^{n}$,
let $\Phi\in\R^{n\times d^{q}}$ be the matrix obtained by applying
$\phi$ to the rows of $\matA$, and consider computing the rank $k$
PCR solution $\Phi$ and $\b$, which we denote by $\x_{{\cal K},k}$.
The corresponding prediction function is $f_{{\cal K},k}(\z)\coloneqq\phi(\z)^{\T}\x_{{\cal K},k}$.
While $\x_{{\cal K},k}$ is likely a huge vector (since $\x_{{\cal K},k}\in\R^{d^{q}}$),
and thus expensive to compute, in kernel PCR we are primarily interested
in having an efficient method to compute $f_{{\cal K},k}(\z)$ given
a 'new' $\z$. We can accomplish this via the kernel trick, as we
now show. 

We assume that $\Phi$ has full row rank (this holds if all data points
are different). Let $\a_{1},\dots,\a_{n}$ be the rows of $\matA$.
As usual with PCR, we have $\x_{{\cal K},k}=\matV_{\Phi,k}\Sigma_{\Phi,k}^{-1}\matU_{\Phi,k}^{\T}\b$.
Since $\matV_{\Phi,k}=\Phi^{\T}\matU_{\Phi,k}\Sigma_{\Phi,k}^{-1}$
we have 
\begin{equation}
f_{{\cal K},k}(\z)=\phi(\z)^{\T}\Phi^{\T}\matU_{\Phi,k}\Sigma_{\Phi,k}^{-2}\matU_{\Phi,k}\b=({\cal K}(\z,\a_{1})\,\cdots\,{\cal K}(\z,\a_{n}))\alpha_{{\cal K},k}\label{eq:kpcr}
\end{equation}
where $\alpha_{{\cal K},k}\coloneqq\matU_{\Phi,k}\Sigma_{\Phi,k}^{-2}\matU_{\Phi,k}^{\T}\b$.
In the above, we used the fact that for any $\x$ and $\z$ we have
$\phi(\x)^{\T}\phi(\z)=(\x^{\T}\z)^{q}={\cal K}(\x,\z)$. Let $\matK\in\R^{n\times n}$
be the\emph{ kernel matrix} (also called \emph{Gram matrix}) defined
by $\matK_{ij}={\cal K}(\a_{i},\a_{j}$). It is easy to verify that
$\matK=\Phi\Phi^{\T}$, so we can compute $\matK$ in $O(n^{2}(d+q))$
(and without forming $\Phi$, which is a huge matrix). We also have
$\matK=\matU_{\Phi}\Sigma_{\Phi}^{2}\matU_{\Phi}^{\T}$ so $\alpha_{k}=\matU_{\matK,k}\Sigma_{K,k}^{-1}\matU_{\matK,k}^{\T}\b$.
Thus, we can compute $\alpha_{{\cal K},k}$ in $O(n^{2}(d+q+n))$
time. Once we have computed $\alpha_{k}$, using~(\ref{eq:kpcr})
we can compute $f_{{\cal K},k}(\z)$ for any $\z$ in $O(ndq)$ time. 

In order to compute an approximate kernel PCR, we introduce a right
sketching matrix $\matR\in\R^{d^{q}\times t}$. Such a matrix $\matR$
is frequently referred to, in the context of kernel learning, as a
randomized feature map. We use the \noun{TensorSketch} feature map~\cite{Pagh13,PP13}.
The feature map is defined as follows. We first randomly generate
$q$ $3$-wise independent hash functions $h_{1},\dots,h_{q}\in\{1,\dots,d\}\to\{1,\dots,t\}$
and $q$ $4$-wise independent sign functions $g_{1},\dots,g_{q}:\{1,\dots,d\}\to\{-1,+1\}$.
Next, we define $H:\{1,\dots,d\}^{q}\to\{1,\dots,t\}$ and $G:\{1,\dots,t\}^{q}\to\{-1,+1\}$:
\[
H(i_{1},\dots,i_{q})\coloneqq h_{1}(i_{1})+\dots+h_{q}(i_{q})\,\mod\,t
\]
\[
G(i_{1},\dots,i_{q})=g_{1}(i_{1})\cdot g_{2}(i_{2})\cdot\dots\cdot g_{q}(i_{q})
\]
 To define $\matR$, we index the rows of $\matR$ by $\{1,\dots,d\}^{q}$
and set row $(i_{1},\dots,i_{q})$ to be equal to $G(i_{1},\dots,i_{q})\cdot\e_{H(i_{1},\dots,i_{q})}$,
where $\e_{j}$ denote the $j$th identity vector. A crucial observation
that makes \noun{TensorSketch} useful, is that via the representation
using $h_{1},\dots,h_{q}$ and $g_{1},\dots,g_{q}$ we can compute
$\matR^{\T}\phi(\z)$ in time $O(q(\nnz{\z}+t\log t))$ (see Pagh~\cite{Pagh13}
for details). Thus, we can compute $\Phi\matR$ in time $O(q(\nnz{\matA}+nt\log t))$.

Consider right sketching PCR on $\Phi$ and $k$ with a \noun{TensorSketch}
$\matR$ as the sketching matrix. The approximate solution is 
\[
\x_{{\cal K},\matR,k}\coloneqq\matR\matV_{\Phi\matR,k}(\Phi\matR\matV_{\Phi\matR,k})^{\pinv}\b=\matR\gamma_{{\cal K},\matR,k}
\]
 where $\gamma_{{\cal K},\matR,k}\coloneqq\matV_{\Phi\matR,k}(\Phi\matR\matV_{\Phi\matR,k})^{\pinv}\b$.
We can compute $\gamma_{\matR,k}$ in $O(q(\nnz{\matA}+nt\log t)+nt^{2})$
time. The predication function is 
\[
f_{{\cal K},\matR,k}(\z)\coloneqq\phi(\z)^{\T}\x_{{\cal K},\matR,k}=(\matR^{\T}\phi(\z))^{\T}\gamma_{{\cal K},\matR,k}
\]
so once we have $\gamma_{{\cal K},\matR,k}$ we can compute $f_{{\cal K},\matR,k}(\z)$
in $O(q(\nnz{\z}+t\log t))$ time. Thus, the method is attractive
from a computational complexity point of view if $t\ll n$ or $d\gg n$
and $d\gg t$. The following theorem bound the excess risk of $\x_{{\cal K},\matR,k}$. 
\begin{thm}
Let $(\nu,\delta)\in(0,1/2)$. Let $\lambda_{1}\geq\dots\geq\lambda_{n}$
denote the eigenvalues of $\matK$. If $\matR$ is a \noun{TensorSketch}
matrix with 
\[
t=\Omega\left(\frac{3^{q}\Trace{\matK}^{2}}{(\lambda_{k}-\lambda_{k+1})^{2}\nu^{2}\delta}\right)
\]
columns, then with probability of at least $1-\delta$
\[
{\cal E}(\x_{{\cal K},\matR,k})\leq{\cal E}(\x_{{\cal K},k})+\frac{(2\nu+\nu^{2})\TNormS{\f}}{n}
\]
where $\f$ is the expected value of $\b$ (recall the statistical
framework in section~\ref{subsec:statistical}).
\end{thm}

Before proving this theorem, we remark that the bound on the size
of the sketch is somewhat disappointing in the sense that it is useful
only if $d\gg n$ (since $\Trace{\matK}$ is likely to be large).
However, this is only a bound, and possibly a pessimistic one. Furthermore,
once the feature expanded data has been embedded in Euclidean space
(via \noun{TensorSketch}), it can be further compressed using standard
Euclidean space transforms like \noun{CountSketch} and subgaussian
maps (this is sometimes referred to as two-level sketching), or compression
can be applied from the left. We leave the task of improving the bound
and exploring additional compression techniques to future research.
\begin{proof}
The square singular values of $\Phi$ are exactly the eigenvalues
of $\matK$, so Theorem~\ref{thm:right_sketching} asserts that the
conclusions of the theorem hold if $\matR^{\T}$ provides $(\epsilon,\delta)$-approximate
Gram matrix for $\Phi$ where $\epsilon=O(\nu(\lambda_{k}-\lambda_{k+1})/\lambda_{1})$.
To that end, we combine the analysis of Avron et al.~\cite{ANW14}
of \noun{TensorSketch} with more recent results due to Cohen et al.~\cite{cohen2015optimal}.
Although not stated as a formal theorem, as part of a larger proof,
Avron et al. show that \noun{TensorSketch} has an OSE-moment property
that together with the results of Cohen et al.~\cite{cohen2015optimal}
imply that indeed the $(\epsilon,\delta)$-approximate Gram property
holds for the specified amounts of columns in $\matR$.
\end{proof}

\section{\label{sec:experiments}Experiments}

In this section we report experimental results, on real data, that
illustrate and support the main results of the paper, and demonstrate
the ability of our algorithms to find appropriately regularized solutions. 

\paragraph{Datasets. }

We experiment with three datasets, two regression datasets (\emph{Twitter
Buzz} and \emph{E2006-tfidf}) and one classification dataset (\emph{Gisette}).

\emph{Twitter Social Media Buzz~\cite{kawala2013predictions}} is
a regression dataset in which the goal is to predict the popularity
of topics as quantified by its mean number of active discussions given
77 predictor variables such as number of authors contributing to the
topic over time, average discussion lengths, number of interactions
between authors etc. We pre-process the data in a manner similar to
previous work \cite{lu2014fast,slawski2017compressed}. That is, several
of the original predictor variables, as well as the response variable
are log-transformed prior to analysis. We then center and scale to
unit norm. Finally, we add quadratic interactions which yielding a
total of $3080$ predictor variables (after preprocessing, the data
matrix is $583250\textrm{-by-}3080$). We used this dataset to explore
only sub-optimality of the objective and constraint satisfaction,
as we have found that the generalization error is very sensitive to
selection of the test set (when splitting a subset of the data to
training and testing).

\emph{E2006-tfidf}~\cite{kogan2009predicting} is regression dataset
where the features are extracted from SEC-mandated financial reports
published annually by a publicly traded company, and the quantity
to be predicted is volatility of stock returns, an empirical measure
of financial risk. We use the standard training-test split available
with the dataset\footnote{We downloaded the dataset from the LIBSVM website, \url{https://www.csie.ntu.edu.tw/~cjlin/libsvmtools/datasets/}.}.
We use this dataset only for testing generalization. The only pre-processing
we performed was subtracting the mean from the response variable,
and reintroducing it when issuing predictions. 

The\emph{ Gisette} dataset is a binary classification dataset that
is a constructed from the MNIST dataset. The goal is to separate the
highly confusable digits '4' and '9'. The dataset has $6000$ data-points,
each having $5000$ features. We use the standard training-test split
available with the dataset (this dataset was downloaded from the same
website as the E2006-tfidf dataset). We convert the binary classification
problem to a regression problem using standard techniques (regularized
least squares classification). We use this dataset only for testing
generalization.

\paragraph{Baselines.}

A first reference are the performance of plain PCR. For small problems,
the dominant right subspace needed to compute the PCR solution can
be computed via MATLAB's dense SVD routine. For larger problems, we
compute the dominant right subspace using a PRIMME~\cite{SM10,WRS17},
a state-of-the-art iterative algorithm for SVD. As additional reference,
we also report results of two alternative algorithms: CLS and the
iterative algorithm of Frostig et al.~\cite{frostig2016principal}.
Both in the discussion, and in the graphs, we refer to the algorithm
Frostig et al.~ as ``Iterative-PCR''. We use the implementation
of Iterative-PCR supplied by the authors,\footnote{\url{https://github.com/cpmusco/fast-pcr}}
for which we used the default parameters, except for the ``tol''
parameter, which we set to $10^{-6}$ instead of the default $10^{-3}$.
We found that the use of tol=$10^{-3}$ produces results that generalize
poorly, while the use of tol=$10^{-6}$ produces much better results.
However, the running time of Iterative-PCR with tol=$10^{-6}$ is
considerably higher the the running time for tol=$10^{-3}$. Iterative-PCR
controls singular vector truncation via a cut-off parameter $\lambda$,
while in our experiments we set $k$ (the number of principal components
that are kept). We achieve this effect by setting $\lambda=(\sigma_{k}^{2}(\matA)+\sigma_{k+1}^{2}(\matA))/2$
(when we report running times, we do not include the time to compute
the singular values). Finally, we remark that based of the documentation,
the algorithm analyzed by Frostig et al.~\cite{frostig2016principal}
does not completely correspond to the default parameters of the implementation
of Iterative-PCR supplied by the authors (e.g., the default parameter
for the ``method'' parameter is ``LANCZOS'', while ``EXPLICIT''
corresponds the algorithm analyzed in \cite{frostig2016principal}). 

\paragraph{Sub-optimality of objective and constraint satisfaction.}

We explore the Twitter Buzz dataset from the optimization perspective,
namely measure the sub-optimality in the objective (vs. PCR) and constraint
satisfaction. Since $n\gg d$, we use left sketching with subgaussian
maps. We perform each experiment five times and report the median
value. Error bars, when present, represent the minimum and maximum
value of five runs. In the top panel, we use a fixed $k=60$ and vary
the sketch size (left sketching only), while in the bottom panel we
vary $k$ and set sketch size to be $s=4k$. The left panel explores
the value of the objective function, appropriately normalized (divided
by $\TNorm{\matA\x_{k}-\b}$ for fixed $k$, and divided by $\TNorm{\b}$
for varying $k$). The right panel explores the regularization effect
by examining the value of the constraints $\TNorm{\matV_{\matA,k}^{\T}\x_{k}}/\TNorm{\b}$. 

In the left panel, we see that value of the objective for the sketched
PCR solution follows the value of objective for the PCR solution.
In general, as the sketch size increases, the variance in the objective
value reduces (top left graph). The normalized value of the constraint
for sketched PCR is rather small (as a reference we note that $\TNorm{\matV_{\matA,k}^{\T}\x_{OLS}}/\TNorm{\b}=0.4165$),
and generally decreases when the sketch size increases (top right
graph), but increases with $k$ for a fixed ratio between $s$ and
$k$ (bottom right graph). Furthermore, the results of sketched PCR
are very similar to the results of iterative PCR (bottom panel), while
running time is considerably shorter (see Table \ref{tab:time}). 

The role of the constraints as a regularizer are illustrated by the
results for CLS (for fixed $k$ we use $t=4k$). As expected, CLS
achieves lower objective value at the price of larger constraint infeasibility.
The values of $\TNorm{\matV_{\matA,k}^{\T}\x_{CLS}}/\TNorm{\b}$ are
much smaller than the OLS value, but much larger than the values for
sketched PCR. Furthermore, it is hard to control the regularization
effect for CLS: when sketch size increases the objective decreases
and the constraint increases (compare to PCR and sketched PCR, top
panel). 

\begin{figure}
\noindent \begin{centering}
\begin{tabular}{ccc}
\includegraphics[width=0.45\textwidth]{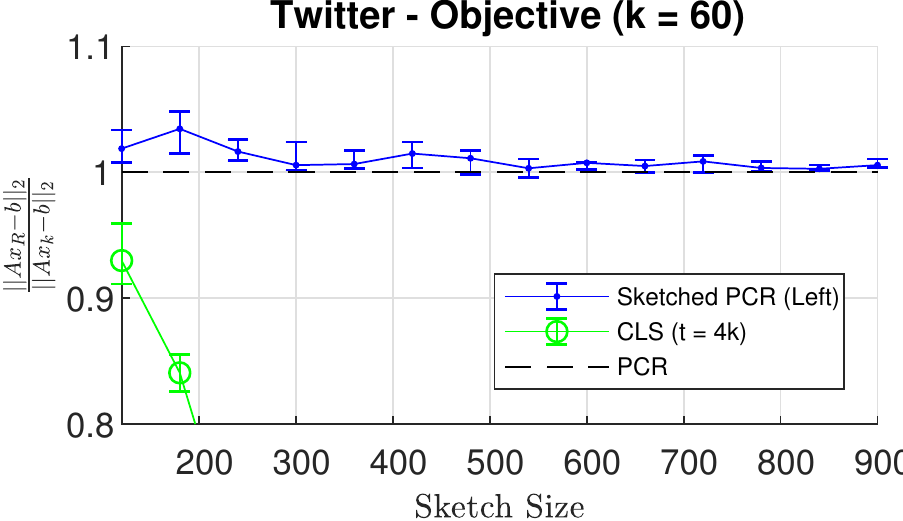} & ~ & \includegraphics[width=0.45\textwidth]{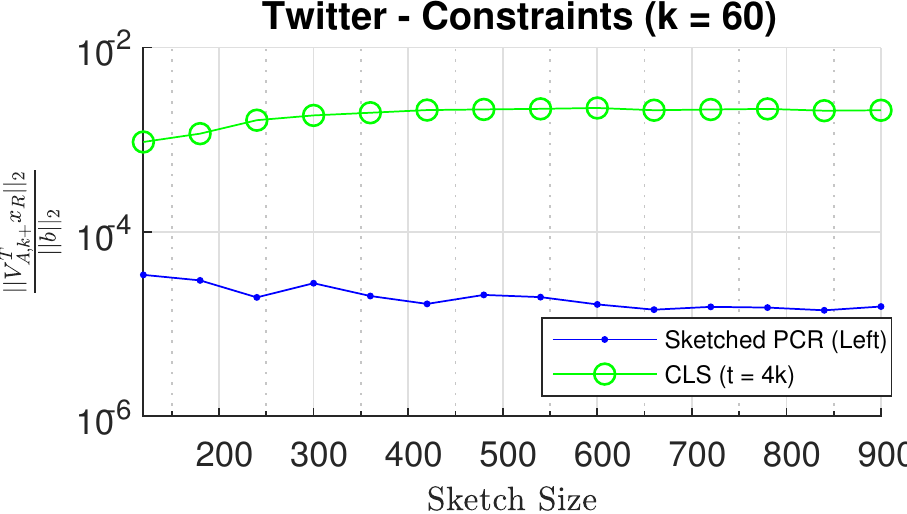}\tabularnewline
\includegraphics[width=0.45\textwidth]{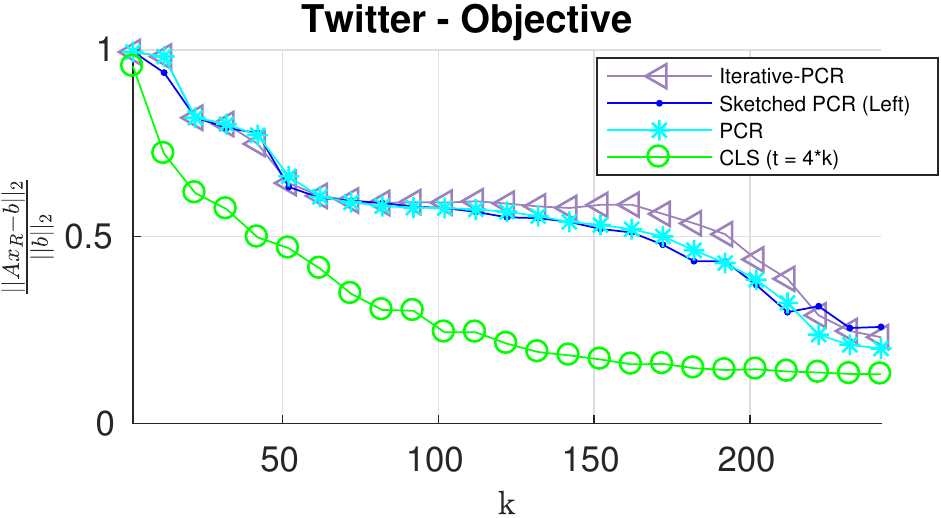} &  & \includegraphics[width=0.45\textwidth]{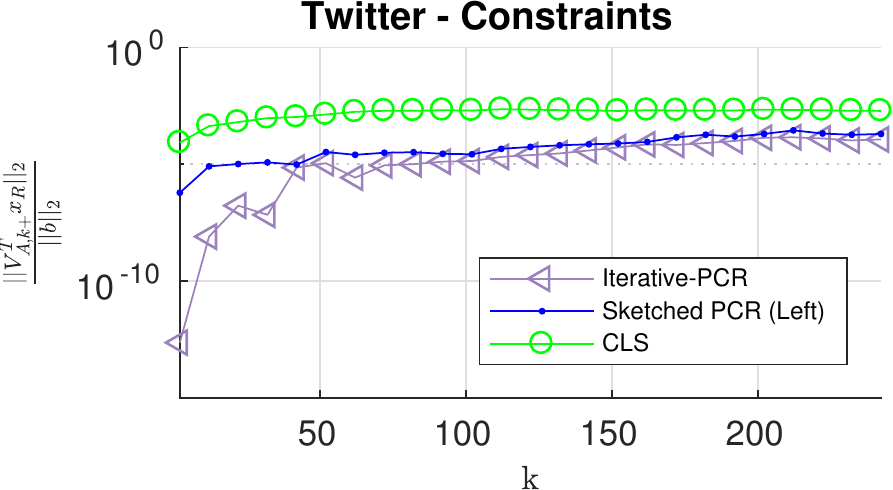}\tabularnewline
\end{tabular}
\par\end{centering}
\caption{\label{fig:twitter}Sub-optimality of objective and constraint satisfaction
for the Twitter Buzz dataset. }
\end{figure}

\paragraph{Generalization results.}

We also explored the prediction error and the tradeoffs between compression
and regularization. We perform each experiment five times and report
the median value. Error bars, when present, represent the minimum
and maximum value of those five runs. 

We report the Mean Square Error (MSE) of predictions for the E2006.tfidf
dataset in Figure~\ref{fig:E2006}. We compare CLS, iterative-PCR,
right sketching and two-sided sketching (the matrix is too large for
exact PCR, and $d\gg n$ so right sketching is more appropriate).
In the left panel we fix $k=600$ and vary the sketch size. The MSE
decreases as the sketch size increases for both sketching methods.
For CLS, initially the MSE decreases and is close to the MSE of the
two sketching methods, but for large sketch sizes the MSE starts to
go up, likely due to decreased level of regularization. We note that
the minimum MSE achieved by CLS is larger than achieved by both sketching
methods. A similar phenomenon is observed when we vary the value of
$k$ in the right panel. 

\begin{figure}
\noindent \begin{centering}
\begin{tabular}{ccc}
\includegraphics[width=0.45\textwidth]{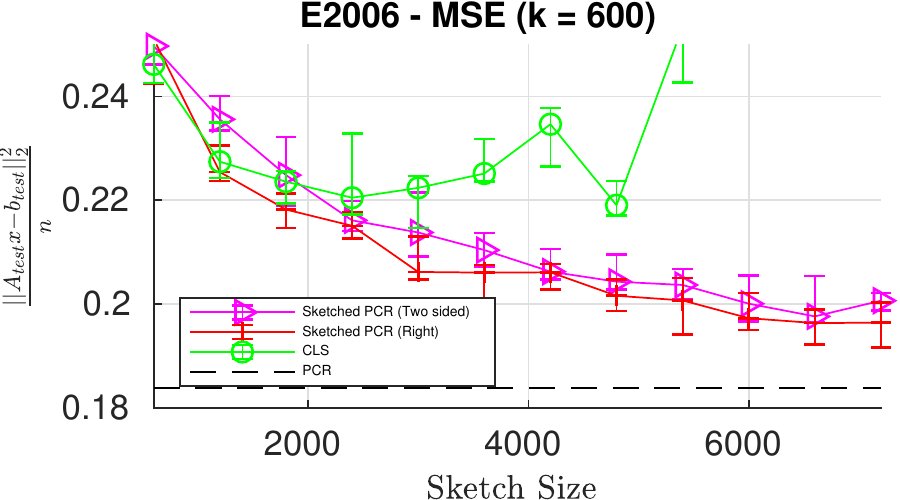} & ~ & \includegraphics[width=0.45\textwidth]{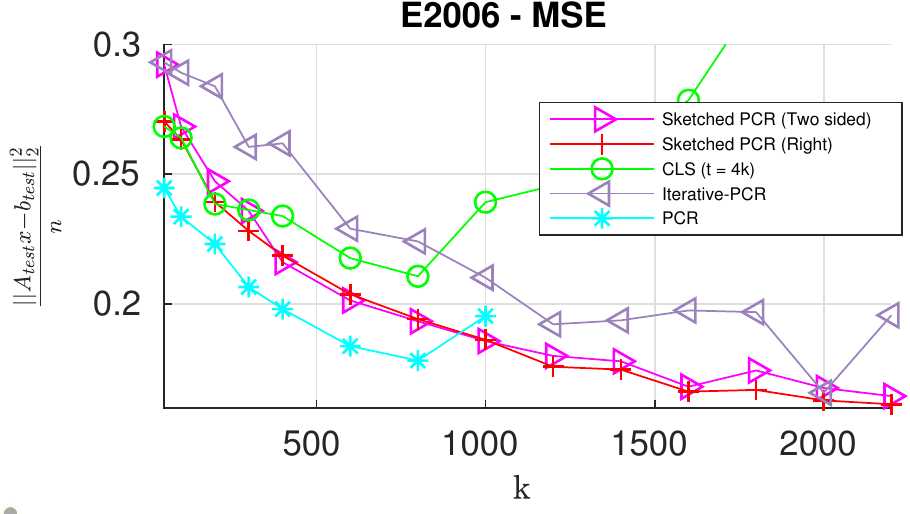}\tabularnewline
\end{tabular}
\par\end{centering}
\caption{\label{fig:E2006}Mean squared error of predictions for the E2006.tfidf
dataset. }
\end{figure}

We report the classification error for the Giesette dataset in Figure~\ref{fig:gisette}.
In the left panel we fix $k=400$ and vary the sketch size. As a reference,
the error rate of exact PCR (with $k=400$) is $2.8\%$ and the error
rate for OLS is 9.3\%. Left sketching has error rate very close to
the error rate of exact PCR, especially when $s$ is large enough.
Right sketching does not perform as well as left sketching, but it
too achieves low error rate for large $s$. For both methods, the
error rate drops as the sketch sizes increase, and the variance reduces.
For CLS the error rate and variance initially drops as the sketch
size increases, but eventually, when sketch size is large enough,
the error rate and the variance increases. This is hardly surprising:
as the sketch size increase, CLS approaches OLS. This is due to the
fact that CLS uses the compression to regularize, and when the sketch
size is large there is little regularization. In the right panel,
we vary the value of $k$ and set $s=4k$ (left sketching) and $t=4k$
(right sketching and CLS). Left sketch and PCR consistently achieve
about the same error rate. For small values of $t$, CLS performs
well, but when $t$ is too large the error starts to increase. In
contrast, right sketching continues to perform well with large values
of $k$. Again, we see that CLS mixes compression and regularization,
and one cannot use a large sketch size and modest amount of regularization
with CLS. 

\begin{figure}
\noindent \begin{centering}
\begin{tabular}{ccc}
\includegraphics[width=0.45\textwidth]{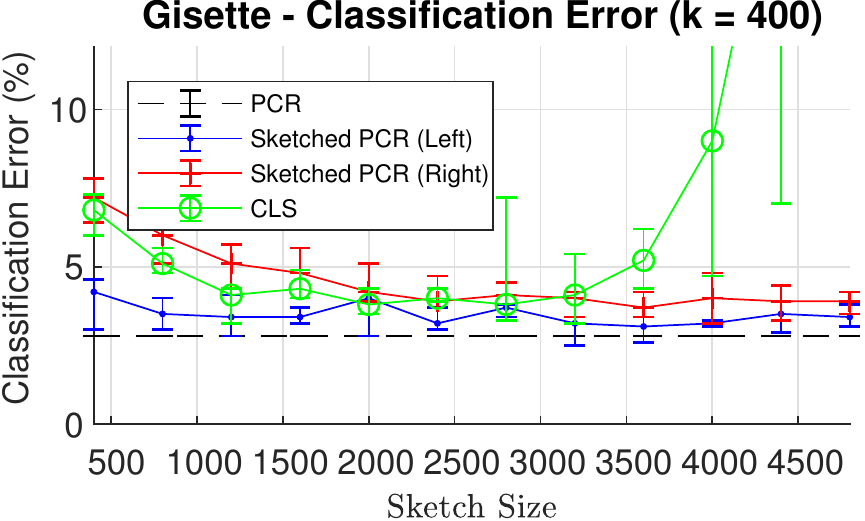} & ~ & \includegraphics[width=0.45\textwidth]{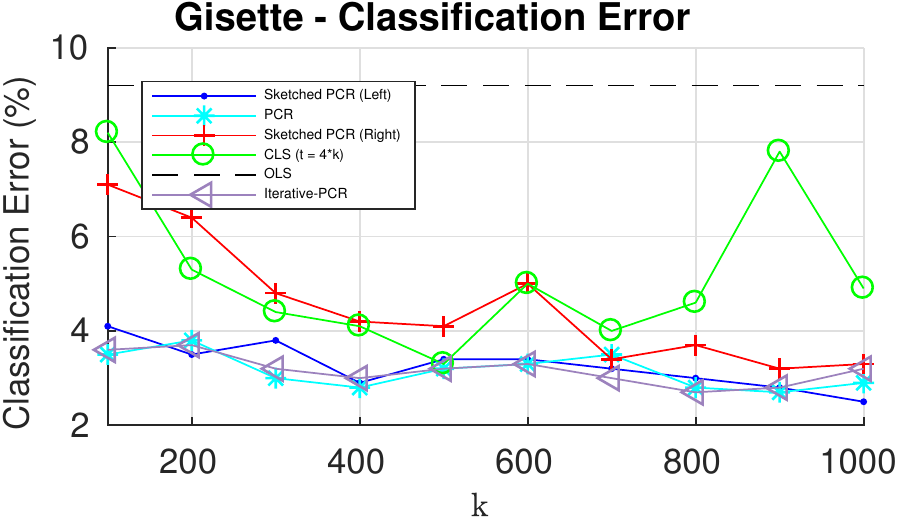}\tabularnewline
\end{tabular}
\par\end{centering}
\caption{\label{fig:gisette}Classification error for the Gisette dataset. }
\end{figure}

\paragraph{Running time. }

In Table~\ref{tab:time} we report a sample of the various running
times of the different algorithms. All experiments were conducted
using MATLAB, although the sketching routines were written in C. Running
times were measured on a machine with a 6-core Intel Xeon Processor
E5-1650 v4 CPU and 128 GB of main memory, running Ubuntu 16.04. For
plain PCR, we report running time using PRIMME, which we ran with
default parameters and no preconditioner. For PRIMME, we cap the number
of iterations at 100,000, and write ``FAIL'' in the table if the
PRIMME failed to convergence within that cap. For iterative PCR we
also report running times when we set tol to the default value, and
reduce the max number of iteration from 40 to 10. This results in
much faster running time, but much degraded generalization (not reported),
e.g. for E2006 the test MSE for Iter-PCR (tol=$10^{-3}$, iter=10)
is 0.32. With respect to running time, iterative-PCR is competitive
with sketched PCR only for E2006, but with worse classification error.
Using PRIMME for PCR is not competitive with sketched PCR. However,
we stress that we experimented with only three datasets, so the comparison
is not comprehensive.

\begin{table}
\centering{}\caption{\label{tab:time}Running times (in seconds). For sketched PCR, we
report in brackets the type of sketching used (left, right, or two-sided).}
{\footnotesize{}}%
\begin{tabular}{|c|c|c|>{\centering}p{2.4cm}|>{\centering}p{2.4cm}|c|}
\hline 
 & {\footnotesize{}CLS ($t=400)$} & {\footnotesize{}PRIMME-PCR} & {\footnotesize{}Iter-PCR }{\footnotesize \par}

{\footnotesize{}(tol=$10^{-3}$, iter=10)} & {\footnotesize{}Iter-PCR (tol=$10^{-6}$)} & {\footnotesize{}Sketched-PCR}\tabularnewline
\hline 
\hline 
{\footnotesize{}Twitter, $k=82$} & {\footnotesize{}21.2} & {\footnotesize{}1730} & {\footnotesize{}742} & {\footnotesize{}4067} & {\footnotesize{}4.7 (left)}\tabularnewline
\hline 
{\footnotesize{}Twitter, $k=152$} & {\footnotesize{}48.9} & {\footnotesize{}5907} & {\footnotesize{}1278} & {\footnotesize{}7759} & {\footnotesize{}8.4 (left) }\tabularnewline
\hline 
{\footnotesize{}E2006, $k=1000$} & {\footnotesize{}100} & {\footnotesize{}14694} & {\footnotesize{}140} & {\footnotesize{}601} & {\footnotesize{}150 (two sided)}\tabularnewline
\hline 
{\footnotesize{}E2006, $k=2000$} & {\footnotesize{}270} & {\footnotesize{}FAIL} & {\footnotesize{}320} & {\footnotesize{}791} & {\footnotesize{}815 (two sided)}\tabularnewline
\hline 
{\footnotesize{}Gisette, $k=400$} & {\footnotesize{}2.0} & {\footnotesize{}497} & {\footnotesize{}7.1} & {\footnotesize{}30.3} & {\footnotesize{}0.2 (left)}\tabularnewline
\hline 
{\footnotesize{}Gisette, $k=1000$} & {\footnotesize{}9.3} & {\footnotesize{}FAIL} & {\footnotesize{}9.3} & {\footnotesize{}39.4} & {\footnotesize{}0.9 (left)}\tabularnewline
\hline 
\end{tabular}
\end{table}

\section{Conclusions and Future work }

In this paper, we studied the use of sketching to accelerate the solution
of PCR and PCP. In particular, for a data matrix $\matA$, we relate
the PCR/PCP solution of $\matA\matR$, where $\matR$ is any dimensionality
reduction matrix, to the PCR/PCP solution of $\matA$. We presented
a notion of approximate PCR/PCP, motivated both from an optimization
perspective and from a statistical perspective, and provide conditions
on $\matR$ that guarantee rigorous theoretical bounds. We then leverage
the aforementioned results to design fast, sketching based, algorithms
for approximate PCR/PCP, and demonstrate empirically the utility of
our proposed algorithms. Throughout, our focus in this paper has been
on algorithms that use the ``sketch-and-solve'' approach. 

There are multiple ways in which the current work can be extended,
and the theoretical results improved. We have presented two notions
of approximation: approximate PCR and approximate PCP. Our results
for approximate PCR use only dimensionality reduction matrices $\matR$
whose number of columns is equal to the target rank. It is natural
to conjecture that the use of dimensionality reduction matrices with
an higher number of columns will lead to stronger PCR bounds, but
we prove only PCP bounds. The underlying reason is that our bounds
for PCR are based on analyzing the distance between the column space
of $\matR$ and the column space of $\matV_{\matA,k}$ . However,
once the number of columns in $\matR$ is different from the number
of columns in $\matV_{\matA,k}$, the definition of $d_{2}(\matR,\matV_{\matA,k})$
is no longer applicable. One possible strategy for analyzing PCR when
$\matR$ has more than $k$ columns might be to use a generalization
of the distance between two subspaces that allows subspaces of different
size; see \cite{ye2016schubert} for such generalizations. Another
crucial component will then be to generalize the Davis-Kahan theorem
to bound such distances. We conjecture it is possible to derive algorithms
that depend on gaps between $\sigma_{k}$ and $\sigma_{k+l}$, where
$l$ is some oversampling parameter, as opposed to the smaller gap
between $\sigma_{k}$ and $\sigma_{k+1}$. We leave this for future
work.

Another interesting direction is in finding other ways to identify
a valid approximate dominant subspace. If we consider the statistical
perspective and inspect Eq.~\ref{eq:bias-var}, we see that all we
need is to find a subspace ${\cal S}\subseteq\range{\matA}$ of rank
$k$ such that $\FNorm{(\matI-\matP_{{\cal S}})\matA}$ is small,
while our theoretical results try to achieve a stronger bound: having
the dominant subspaces align. One possible way for finding such a
$\calS$ is using so-called Projection-cost Preserving Sketches~\cite{CohenEMMP15}.
We leave this for future work.

\subsection*{Acknowledgments.}

This research was supported by the Israel Science Foundation (grant
no. 1272/17) and by an IBM Faculty Award.

\bibliographystyle{plain}
\bibliography{kpcr}

\appendix

\section{Bias-Variance Decomposition for ${\cal E}(\protect\x_{\protect\matR})$}

The following appears, without proof, in~\cite{Slawski17}. For completeness,
we include a proof.
\begin{claim}
\label{claim:excess_risk_r-1}The excess risk of $\x_{\matR}$ can
be bounded as follows: 
\[
{\cal E}(\x_{\matR})=\underset{{\cal B}(\x_{\matR})}{\underbrace{\frac{1}{n}\TNormS{\left(\matI-\matP_{\matA\matR}\right)\matA\x^{\star}}}}+\underset{{\cal V}(\x_{\matR})}{\underbrace{\sigma^{2}\frac{\rank{\matA\matR}}{n}}}\,.
\]
\end{claim}

\begin{proof}
The column space of $\matA\matR$ is contained in the column space
of $\matA$, so we have $\matP_{\matA\matR}=\matP_{\matA\matR}\matP_{\matA}$.
We now observe,
\begin{eqnarray*}
{\cal E}(\x_{\matR}) & = & \frac{1}{n}\Expect{\TNormS{\matA\x_{\matR}-\matA\x^{\star}}}\\
 & = & \frac{1}{n}\Expect{\TNormS{\matP_{\matA\matR}\b-\matP_{\matA}\f}}\\
 & = & \frac{1}{n}\Expect{\TNormS{\matP_{\matA\matR}\f-\matP_{\matA}\f}}+\frac{1}{n}\Expect{\TNormS{\matP_{\matA\matR}\xi}}\\
 & = & \frac{1}{n}\Expect{\TNormS{\matP_{\matA\matR}\matP_{\matA}\f-\matP_{\matA}\f}}+\sigma^{2}\frac{\rank{\matA\matR}}{n}\\
 & = & \frac{1}{n}\TNormS{\left(\matI-\matP_{\matA\matR}\right)\matA\x^{\star}}+\sigma^{2}\frac{\rank{\matA\matR}}{n}
\end{eqnarray*}
where in the third line we used the fact that expected value of $\xi$
is $0$, and in the fourth line we used that fact that for any matrix
$\mat M$ and random vector $\y$ with independent entries with 0
mean and $\sigma^{2}$ variance we have $\Expect{\y^{\T}\matM\y}=\Trace{\matM}$.
\end{proof}

\end{document}